\documentclass[a4paper,10pt]{article}

\usepackage{amsfonts,amsthm,amssymb,amsmath,amscd,ascmac}

\usepackage{latexsym}
\usepackage{overpic}

\newtheorem{lem}{Lemma}[section]
\newtheorem{prop}{Proposition}[section] 
\newtheorem{cor}{Corollary}[section] 
\newtheorem{rem}{Remark}[section] 
\newtheorem{tm}{Theorem}[section] 
\newcommand{\ZZ}{\mathbb{Z}}
\newcommand{\QQ}{\mathbb{Q}}
\newcommand{\RR}{\mathbb{R}}
\newcommand{\CC}{\mathbb{C}}
\newcommand{\PP}{\mathbb{P}}

\newcommand{\ee}{\mathbf{e}} \newcommand{\ff}{\mathbf{f}}
\newcommand{\Mon}{\mathbf{Mon}} \newcommand{\Ref}{\mathbf{Ref}}

\newcommand{\Sp}{\mathbf{Sp}} \newcommand{\SO}{\mathbf{SO}}
\newcommand{\GL}{\mathbf{GL}} \newcommand{\SL}{\mathbf{SL}}
\title%[Picard-Vessiot groups of Lauricella's systems $E_C$]
{Picard-Vessiot groups of Lauricella's hypergeometric systems $E_C$ and 
Calabi-Yau varieties arising integral representations}

\author{Yoshiaki Goto and Kenji Koike}  

\begin{document}
\maketitle
%%%%%%%%%%%%%%%%%%%%%%%%%%%%%%%%%%%%%%%%%%%%%%%%%% ABSTRACT
\begin{abstract}
We study the Zariski closure of the monodromy group $\Mon$ of Lauricella's 
hypergeometric function $F_C$. If the identity component $\Mon^0$ acts irreducibly, 
then $\overline{\Mon} \cap \SL_{2^n}(\CC)$ must be one of classical groups 
$\SL_{2^n}(\CC), \SO_{2^n}(\CC)$ and $\Sp_{2^n}(\CC)$. 
We also study Calabi-Yau varieties arising from integral 
representations of $F_C$.
\end{abstract}
%%%%%%%%%%%%%%%%%%%%%%%%%%%%%%%%%%%%%%%%%%%%%%%%%%%%%%%%%%%%%%%%%%%%%%%%%%%%%%%
\section{Introduction}
In \cite{BH}, Beukers and Heckman studied the monodromy of the generalized hypergeometric 
function ${}_{n+1}F_n$ from a viewpoint of differential Galois theory. They determined 
the differential Galois group called the Picard-Vessiot group (for Fuchsian equations, 
which is given by the Zariski closure of the monodromy group), 
and parameters for which the monodromy group is finite. 
In this paper, applying their method for results in \cite{Go2} and \cite{GM2}, 
we study the Zariski closure of the monodromy group of Lauricella's hypergeometric 
function 
\begin{align*}
  F_C(a,b,c;x) =
  \sum_{m_1 ,\ldots ,m_n =0} ^{\infty } 
  \frac{(a)_{m_1 +\cdots +m_n} (b)_{m_1 +\cdots +m_n}}
  {(c_1)_{m_1}\cdots (c_n)_{m_n} m_1 ! \cdots m_n !} x_1^{m_1} \cdots x_n^{m_n},
\end{align*} 
and we also study Calabi-Yau varieties arising from integral representations of $F_C$.
\\  \indent 
Lauricella's hypergeometric function $F_C$, together with  $F_A, F_B$ and $F_D$, was 
introduced by Appell and Lauricella in the 19th century as generalizations of the 
Gauss hypergeometric function ${}_2F_1$. In the case of two variables, Lauricella's 
$F_A, F_B, F_C$ and $F_D$ are called Appell's $F_2, F_3, F_4$ and $F_1$ respectively. 
The monodromy of these functions have been studied by many authors. In \cite{Sa}, Sasaki 
showed that Picard-Vessiot groups for $F_i \ (i=2,3,4)$ and $F_D$ are general linear 
groups for general parameters. 
Deligne and Mostow gave a list of parameters of Lauricella's $F_D$ that 
produce complex ball uniformizations, and concrete examples of non-arithmetic subgroup of 
unitary groups in \cite{DM}. Recently the structure of monodromy group for $F_C$ was 
studied in \cite{Go2} and \cite{GM2}. According to them, the monodromy group $\Mon$ for 
$F_C$ is generated by $M_0, M_1, \dots, M_n$ where $M_0$ is a reflection in the sense 
of \cite{BH} and $M_1,\dots, M_n$ form an Abelian subgroup. Applying results in \cite{BH}, 
we can classify the Zariski closure $\overline{\Mon}$ as in the case of the generalized 
hypergeometric function ${}_{n+1}F_n$. If the identity component $\Mon^0$ acts irreducibly, 
then $\overline{\Mon} \cap \SL_{2^n}(\CC)$ must be one of classical groups 
$\SL_{2^n}(\CC), \SO_{2^n}(\CC)$ and $\Sp_{2^n}(\CC)$ (Theorem \ref{main}). 
To study irreducibility conditions of the identity component, we introduce 
the reflection subgroup $\Ref \subset \Mon$, generated by $gM_0g^{-1} \ (g \in \Mon)$. 
In Theorem \ref{main2}, we give the necessary and sufficient condition for the 
irreducibility of $\Ref$ in terms of parameters. It is simply that
at most one of $\gamma_1 ,\ldots ,\gamma_n ,\alpha \beta^{-1}$ is $-1$  
in addition to irreducibility conditions for $\Mon$ in Proposition \ref{prop-irr}. 
The proof is based on ideas in a work of Kato for Appell's $F_4$ (\cite{Ka}). 
Moreover we prove that if the action of the identity component $\Ref^0$ is reducible 
and $\Ref$ is irreducible, then $\Ref$ and $\Mon$ is finite (Theorem \ref{main3}).   
\\ \indent
%%%%%%%%%%%%%
In the last section, we study double coverings $V(x)$ of projective spaces associated to  
integral representations of $F_C(a,b,c;x)$ with $a=b=1/2, c_k=1$.
It is known that the monodromy group for hyperelliptic curves is arithmetic, that is, 
finite index in $\Sp_{2g}(\ZZ)$. Since a period integral for a hyperelliptic curve is 
given by Lauricella's function $F_D$, our varieties are regarded as the counterpart of 
hyperelliptic curves. By the results of the former part, we see that the Zariski closure 
of the monodromy group for $a=b=1/2, c_k=1$ is the symplectic or orthogonal group. 
It is interesting to study arithmeticy of these group. In the case of $n=2$, it is well 
known that Appell's hypergeometric function $F_4(1/2,1/2,1,1;x_1,x_2)$ is a products of 
Gauss's hypergeometric functions. We show that $V(x)$ is in fact a product Kummer surface, 
and the monodromy group contains $\Gamma(2) \times \Gamma(2)$ as a subgroup of index 2. 
In the case of $n=3$, 
we have double octic Calabi-Yau varieties $\widetilde{V}$ of Euler number $128$
by resolving singularities. 
For computation of Euler and Hodge numbers, we use methods in \cite{CS} and \cite{CvS}. 
For $n \geq 4$, we do not know if there are crepant resolutions of $V(x)$.   
%%%%%%%%%%%%%%%%%%%%%%%%%%%%%%%%%%%%%%%%%%%%%%%%%%%%%%%%%%%%%%%%%%%%%%%%%%%%%%%%%%
\section{Monodromy of the system $E_C$}
%%%%%%%%%%%%%%%%%%%%%%%%%%%%%%%%%%%%%%%%%%%%%%%%%%%%%%%%%%%%%%%%%%%%%%%%%%%%%%%%%%
\subsection{Lauricella's hypergeometric function $F_C$} 
%%%%%%%%%%%%%%%%%%%%%%%%%%%%%%%%%%%%%%%%%%%%%%%%%%%%%%%%%%%%%%%%%%%%%%%%%%%%%%%%%%
Lauricella's hypergeometric function $F_C$ of $n$ variables $x_1, . . . , x_n$ is
\begin{align*}
  F_C(a,b,c;x) =
  \sum_{m_1 ,\ldots ,m_n =0} ^{\infty } 
  \frac{(a)_{m_1 +\cdots +m_n} (b)_{m_1 +\cdots +m_n}}
  {(c_1)_{m_1}\cdots (c_n)_{m_n} m_1 ! \cdots m_n !} x_1^{m_1} \cdots x_n^{m_n} ,
\end{align*}
where $x=(x_1 ,\ldots ,x_n)$, $c=(c_1 ,\ldots ,c_n)$, 
$c_1 ,\ldots ,c_n \not\in \{ 0,-1,-2,\ldots \}$, and $(c_1)_{m_1}=\Gamma (c_1+m_1)/\Gamma (c_1)$. 
This series converges in the domain 
$\left\{ (x_1 ,\ldots ,x_n ) \in \CC^n  \ \middle| \ \sum_{k=1}^{n} \sqrt{|x_k|} <1  \right\}$.
In the case of $n=2$, the series $F_C(a,b,c;x)$ is called 
Appell's hypergeometric series $F_4(a,b,c_1,c_2;x_1,x_2)$. 
%%%%%%%%%%%%%%%%%%%%%%%%%%%%%%%%%%%
Let $\partial_k \ (k =1, \ldots , n)$ be the partial differential operator with respect 
to $x_k$. We set $\theta_k =x_k \partial_k$, $\theta =\sum_{k=1}^n \theta_k$.  
Lauricella's $ F_C (a,b,c;x)$ satisfies differential equations 
$$
\left[ \theta_k (\theta_k+c_k-1)-x_k(\theta +a)(\theta +b)  \right] f(x)=0, 
\quad k =1, \ldots , n.
$$ 
The system generated by them is 
called Lauricella's hypergeometric system $E_C (a,b,c)$ of differential equations. 
%% Proposition %%%%%%%%%%%%%%%%%%%%%%%%%%%%%%%%%%%%%%%%%%%%%%%%%%%%%%%%%%%%%%%%%%
\begin{prop}[(\cite{HT})]
The system $E_C (a,b,c)$ is a holonomic system of rank $2^n$ with the singular locus 
\begin{align*}
  S:= \left( \prod_{k=1}^n x_{k} \cdot R(x)=0 \right) \subset \CC^n ,\quad
  R(x_1 ,\ldots ,x_n):=\prod_{\epsilon_1 ,\ldots ,\epsilon_n =\pm 1} 
  \left( 1+\sum_{k=1}^n \epsilon_k \sqrt{x_k} \right) . 
\end{align*}
\end{prop}
%%%%%%%%%%%%%%%%%%%%%%%%%%%%%%%%%%%%%%%%%%%%%%%%%%%%%%%%%%%%%%%%%%%%%%%%%%%%%%%%
In \cite{AK} and \cite{Ki}, 
an integral representation of $F_C(a,b,c;x)$ with generic parameters 
is given in terms of the twisted cycles.   
\begin{prop}[(\cite{AK}, \cite{Ki})]
  For sufficiently small positive real numbers $x_1 ,\ldots ,x_n$, 
  if $c_1 ,\ldots ,c_n ,a-\sum c_k \not\in \ZZ$, 
  then $F_C (a,b,c;x)$ admits the following integral representation: 
  \begin{align*}
    F_C (a,b,c ;x)  
    =&\frac{\Gamma (1-a)}{\prod_k \Gamma (1-c_k)\cdot \Gamma (\sum_k c_k -a-n+1)}  \\
    & \cdot \int_{\Delta} \prod_k t_k ^{-c_k} \cdot (1-\sum_k t_k )^{\sum_k c_k -a-n} 
    \cdot \left( 1-\sum_k \frac{x_k}{t_k} \right) ^{-b} dt_1 \wedge \cdots \wedge dt_n , 
  \end{align*}
  where $\Delta$ is the twisted cycle made by an $n$-simplex, 
  in Sections 3.2 and 3.3 of \cite{AK}. 
\end{prop}
For our applications, we show that $F_C$ has an Euler-type integral representation 
even if $c_i$'s are positive integers. 
\begin{prop}
  We assume $c_1,\ldots,c_n \in \ZZ_{\geq 1}$. 
  Let $\epsilon$ and $x_k$ be small positive real numbers such that
  $$
  0<\epsilon <\frac{1}{n+1},\quad
  0<x_k <\frac{\epsilon^2}{n}. 
  $$
  Then the integration on the direct product 
  $$
  C_{\epsilon}^n : |t_1|=|t_2|=\cdots =|t_n|=\epsilon
  $$
  of circles expresses $F_C$: 
  \begin{align*}
    F_C(a,b,c;x) 
    =\frac{(-1)^{n+\sum_k c_k}}{(2\pi \sqrt{-1})^n} \frac{\Gamma (1-a)\prod_k \Gamma(c_k)}{\Gamma(1-a-n+\sum_k c_k)}
    \int_{C_{\epsilon}^n } \prod_k t_k ^{-c_k} \cdot (1-\sum_k t_k )^{\sum_k c_k -a-n} 
    \cdot \left( 1-\sum_k \frac{x_k}{t_k} \right) ^{-b} dt ,
  \end{align*}
  where $dt=dt_1 \wedge \cdots \wedge dt_n$. 
\end{prop}
\begin{proof}
  By the assumption, if $(t_1,\ldots,t_n)\in C_{\epsilon}^n$, we have 
  \begin{align*}
    &\left| \sum_k \frac{x_k}{t_k} \right| 
    \leq \sum_k \frac{|x_k|}{|t_k|} 
    <\sum_k \frac{\epsilon^2}{n} \cdot \frac{1}{\epsilon} =\epsilon <1 ,\\
    &\left| \sum_k t_k \right|
    \leq \sum_k |t_k| =n\epsilon <\frac{n}{n+1}<1. 
  \end{align*}
  Thus the series
  \begin{align*}
  \left( 1-\sum_k \frac{x_k}{t_k} \right) ^{-b} 
  &=\sum_{m_1,\dots,m_n} \frac{(b)_{m_1+\dots+m_n}}{\prod_k m_k!} 
  \prod_k \left( \frac{x_k}{t_k} \right)^{m_k} , \\
  \left( 1-\sum_k t_k \right)^{\sum_k c_k -a-n}
  &=\sum_{p_1,\dots,p_n} \frac{(a+n-\sum c_k)_{p_1+\dots+p_n}}{\prod_k p_k!} 
  \prod_k t_k^{p_k}     
  \end{align*}
  uniformly converge on $C_{\epsilon}^n$, and hence we have
  \begin{align*}
    &\int_{C_{\epsilon}^n } \prod_k t_k ^{-c_k} \cdot (1-\sum_k t_k )^{\sum_k c_k -a-n} 
    \cdot \left( 1-\sum_k \frac{x_k}{t_k} \right) ^{-b} dt \\
    &=\sum_{m_1,\dots,m_n} \frac{(b)_{m_1+\dots+m_n}}{\prod_k m_k!} \prod_k x_k^{m_k}
    \int_{C_{\epsilon}^n } \prod_k t_k ^{-c_k-m_k} \cdot (1-\sum_k t_k )^{\sum_k c_k -a-n} dt \\
    &=\sum_{m_1,\dots,m_n} \sum_{p_1,\dots,p_n}
    \frac{(a+n-\sum c_k)_{p_1+\dots+p_n}(b)_{m_1+\dots+m_n}}{\prod_k p_k!\prod_k m_k!} 
    \prod_k x_k^{m_k} 
    \int_{C_{\epsilon}^n } \prod_k t_k^{p_k-c_k-m_k} dt .
  \end{align*}
  By the residue theorem, only the terms with $p_k =c_k+m_k-1$ survive. 
  If $p_k =c_k+m_k-1$, then 
  \begin{align*}
    \frac{(a+n-\sum c_k)_{p_1+\dots+p_n}}{\prod_k p_k!}
    =\frac{(a+n-\sum c_k)_{\sum_k c_k +\sum_k m_k -n}}{\prod_k (c_k+m_k-1)!}
    =\frac{\Gamma (a+\sum_k m_k)}{\Gamma(a+n-\sum c_k) \cdot \prod_k \Gamma (c_k+m_k)} .
  \end{align*}
  Thus we obtain
  \begin{align*}
    &\int_{C_{\epsilon}^n } \prod_k t_k ^{-c_k} \cdot (1-\sum_k t_k )^{\sum_k c_k -a-n} 
    \cdot \left( 1-\sum_k \frac{x_k}{t_k} \right) ^{-b} dt \\
    &=(2\pi \sqrt{-1})^n
    \sum_{m_1,\dots,m_n} \frac{(b)_{m_1+\dots+m_n}}{\prod_k m_k!} \prod_k x_k^{m_k}
    \cdot \frac{\Gamma (a+\sum_k m_k)}{\Gamma(a+n-\sum c_k) \cdot \prod_k \Gamma (c_k+m_k)} \\
    &=(2\pi \sqrt{-1})^n \frac{\Gamma (a)}{\Gamma(a+n-\sum c_k) \prod_k \Gamma(c_k)}
    \sum_{m_1,\dots,m_n} \frac{(a)_{\sum m_k} (b)_{\sum m_k}}{\prod_k (c_k)_{m_k} \prod_k m_k!} 
    \prod_k x_k^{m_k} \\
    &=(2\pi \sqrt{-1})^n \frac{\Gamma (a)}{\Gamma(a+n-\sum c_k) \prod_k \Gamma(c_k)}
    \cdot F_C (a,b,c;x) . 
  \end{align*}
  By using the reflection formula, we conclude the proposition. 
\end{proof}

\begin{rem}
  Roughly, $C_{\epsilon}^n$ can be regarded as 
  the limit of $\prod_k (1-e^{-2\pi i c_k}) \cdot \Delta$ as $c_k$'s to integers. 
\end{rem}
%%%%%%%%%%%%%%%%%%%%%%%%%%%%%%%%%%%%%%%%%%%%%%%%%%%%%%%%%%%%%%%%%%%%%%%%%%%%%%%%%%
%%%%%%%%%%%%%%%%%%%%%%%%%%%%%%%%%%%%%%%%%%%%%%%%%%%%%%%%%%%%%%%%%%%%%%%%%%%%%%%%%%
\subsection{Monodromy representation}
The monodromy representation of $F_C$ is expressed in terms of the twisted homology 
groups in \cite{Go2}, 
and clear representation matrices of circuit transformations are 
obtained in \cite{GM2}. 
Here we briefly review results in \cite{Go2} and \cite{GM2}.
Let $X$ be the complement of the singular locus $S$ of the system $E_C(a,b,c)$. 
Put $\dot{x}=\left( \frac{1}{2n^2},\ldots ,\frac{1}{2n^2} \right) \in X$. 
Let $\rho_0, \rho_1 ,\ldots ,\rho_n$ be loops in $X$ so that 
\begin{itemize}
\item $\rho_0$ turns the hypersurface $(R(x)=0)$ around the point 
  $\left( \frac{1}{n^2},\ldots, \frac{1}{n^2} \right)$, positively, 
\item $\rho_k \ (k=1,\dots , n)$ turns the hyperplane $(x_k=0)$, positively. 
\end{itemize}
For explicit definitions of them, see \cite{Go2}. 
%%% Proposition %%%%%%%%%%%%%%%%%%%%%%%%%%%%%%%%%%%%%%%%%%%%%%%%%%%%%%%%%%%%%%%%%%
\begin{prop}[(\cite{Go2})]
  The loops $\rho_0 ,\rho_1 ,\ldots ,\rho_n$ generate 
  the fundamental group $\pi_1 (X,\dot{x})$. 
  Moreover, if $n\geq 2$, then they satisfy the following relations:
  \begin{align*}
    \rho_i \rho_j =\rho_j \rho_i \quad (i,j=1,\dots , n) ,\quad
    (\rho_0 \rho_k)^2 =(\rho_k \rho_0)^2 \quad (k=1,\dots , n).
  \end{align*}
\end{prop} 
%%%%%%%%%%%%%%%%%%%%%%%%%%%%%%%%%%%%%%%%%%%%%%%%%%%%%%%%%%%%%%%%%%%%%%%%%%%%%%%%
Let $\mathcal{M}_i$ be the circuit transformation corresponding to the loop 
$\rho_i$ ($i=0,\ldots ,n$). To write down representation matrices of $\mathcal{M}_i$, 
it is convenient to regard $\CC^{2^n}$ as $\CC^2 \otimes \cdots \otimes \CC^2$ and 
take a basis 
\[
\ee_{i_1, \dots, i_n} = \ee_{i_1} \otimes \cdots \otimes \ee_{i_n}, \qquad
\ee_0 = \begin{pmatrix} 1 \\ 0 \end{pmatrix}, \ 
\ee_1 = \begin{pmatrix} 0 \\ 1 \end{pmatrix}. 
\]
We align them in the pure lexicographic order of 
indices $I = (i_1, \dots, i_n) \in \{0,1\}^n$:
$$
(0,\dots,0),\ (1,0,\dots,0), \ (0,1,\dots,0),\  (1,1,\dots,0),\ 
(0,0,1,\dots,0),\ \dots, (1,\dots,1).
$$
We define the tensor product $A\otimes B$ of matrices 
$A$ and $B=(b_{ij})_{\substack{1\le i\le r\\ 1\le j\le s}}$ as 
$$A\otimes B=\begin{pmatrix}
A\, b_{11} & A\, b_{12} &\cdots & A\, b_{1s} \\
A\, b_{21} & A\, b_{22} &\cdots & A\, b_{2s} \\
\vdots  &\vdots   &\ddots & \vdots  \\
A\, b_{r1} & A\, b_{r2} &\cdots & A\, b_{rs} \\
\end{pmatrix}, 
$$
and we put
\begin{align*}
  \alpha =\exp (2\pi \sqrt{-1} a),\quad 
  \beta =\exp (2\pi \sqrt{-1} b),\quad 
  \gamma_k =\exp (2\pi \sqrt{-1} c_k) \qquad (k=1,\ldots ,n).
\end{align*}
Regarding $\alpha, \beta, \gamma_k$ just as symbols, we define an isomorphism
${}^{\vee} : \QQ(\alpha, \beta, \gamma_k) \rightarrow \QQ(\alpha, \beta, \gamma_k)$
of a rational function field by $\alpha \mapsto \alpha^{-1}$, $\beta \mapsto \beta^{-1}$, 
$\gamma_k \mapsto \gamma_k^{-1}$. If $a, b, c_k \in \RR$, then ${}^{\vee}$ is 
nothing but the complex conjugation. With these notations, we have
%%%%%%%%%%%%%%%%%%%%%%%%%%%%%%%%%%%%%%%%%%%%%%%%%%%%%%%%%%%%%%%%%%%%%%%%%%%%%%%%%%%%%%%
%% Proposition %%%%%%%%%%%%%%%%%%%%%%%%%%%%%%%%%%%%%%%%%%%%%%%%%%%%%%%%%%%%%%%%%%%%%%%%
\begin{prop}[(\cite{Go2}, \cite{GM2})]\label{prop-rep-matrix}
  For a certain basis of the solution space to $E_C(a,b,c)$, 
  the representation matrix $M_i$ of $\mathcal{M}_i$ $(i=0, \dots ,n)$
  is written as follows. 
  For $k=1,\dots, n$, we have 
  \begin{align*}
    M_k =E_2 \otimes \cdots \otimes E_2 \otimes \underset{k\textrm{-th}}{G_k} 
    \otimes E_2 \otimes \cdots \otimes E_2 ,
    \quad 
    G_k =\left(
      \begin{array}{cc}
        1 & -\gamma_k^{-1} \\ 0 & \gamma_k^{-1}
      \end{array}
      \right).
  \end{align*}
  The matrix $M_0$ is written as 
  \begin{align*}
    M_0 =E_{2^n} -N_0 ,\quad 
    N_0 ={}^t (\mathbf{0}, \dots ,\mathbf{0},v) ,
  \end{align*}
  where $v \in \CC^{2^n}$ is a column vector whose $I$-th entry is 
  \begin{align*}
    \left\{ 
      \begin{array}{ll}
        (-1)^n \frac{(\alpha-1)(\beta-1)\prod_{k=1}^n \gamma_k}{\alpha \beta}& 
        (I=(0,\dots ,0)) ,\\
        (-1)^{n+|I|} \frac{(\alpha \beta +(-1)^{|I|} \prod_{k=1}^n \gamma_k^{i_k})
          \prod_{k=1}^n \gamma_k^{1-i_k}}{\alpha \beta}& (I\neq(0,\dots ,0)) .
      \end{array}
    \right.
  \end{align*}
  Further, the intersection matrix $H=(H_{I,I'})$ defined as 
  \begin{align*}
    H_{I,I'}=\left\{
      \begin{array}{ll}
        \displaystyle
        \prod_{k=1}^n (-\gamma_k)^{i'_k}(1-\gamma_k)^{1-i_k-i'_k} 
        \cdot \frac{\alpha -1}{\alpha -\prod_{k=1}^n  \gamma_k} 
        & (I\cdot I'=(0,\dots ,0)), \\
        \displaystyle
        \frac{\alpha \beta +(-1)^{|I\cdot I'|}\prod_{k=1}^n \gamma_k^{i_k i'_k}}
        {( \alpha -\prod_{k=1}^n  \gamma_k ) (\beta -1)}
        \cdot \prod_{k=1}^n (-\gamma_k)^{i'_k(1-i_k)} (1-\gamma_k)^{(1-i_k)(1-i'_k)} 
        & (\textrm{otherwise}) 
      \end{array}
      \right.
  \end{align*}
  satisfies ${}^t M_i \cdot H \cdot M_i^{\vee} =H$ and ${}^t H =(-1)^n H^{\vee}$. 
  Here, $I\cdot I'=(i_1 i'_1 ,\dots ,i_n i'_n )$, $|I| = i_1 + \cdots + i_n$ for 
  $I=(i_1,\dots ,i_n)$ and $I'=(i'_1,\dots ,i'_n)$. 
\end{prop}
%%%%%%%%%%%%%%%%%%%%%%%%%%%%%%%%%%%%%%%%%%%%%%%%%%%%%%%%%%%%%%%%%%%%%%%%%%%%%%%%%%%%%
%% Remark %%%%%%%%%%%%%%%%%%%%%%%%%%%%%%%%%%%%%%%%%%%%%%%%%%%%%%%%%%%%%%%%%%%%%%%%%%%
\begin{rem} \label{rem-M0}
 {\rm (1)} \ In particular, $M_0$ is lower triangular and $M_1,\dots ,M_n$ 
are upper triangular.
\\
{\rm (2)} Proposition \ref{prop-rep-matrix} implies that  
$\ee_{1,\dots,1} =\ee_1 \otimes \cdots \otimes \ee_1$ is an eigenvector of $M_0$, 
that is,  
$$
M_0 \ee_{1,\dots,1} = \delta_0 \ee_{1,\dots,1} ,\quad 
\delta_0  = (-1)^{n+1} \frac{\gamma_1 \cdots \gamma_n}{\alpha \beta} .
$$
In \cite{GM2}, it is also shown that 
the eigenspace of $M_0$ with eigenvalue $1$ is expressed as
$$
\ker N_0 
=\{ w\in \CC^{2^n} \mid {}^t w H \ee_{1,\dots,1} =0\}.
$$ 
The matrix $M_0$ is a ``reflection'' defined later, with the special eigenvalue $\delta_0$. 
\end{rem}
%%%%%%%%%%%%%%%%%%%%%%%%%%%%%%%%%%%%%%%%%%%%%%%%%%%%%%%%%%%%%%%%%%%%%%%%%%%%%%%%%%%%%%%%%%
%%%%%%%%%%%%%%%%%%%%%%%%%%%%%%%%%%%%%%%%%%%%%%%%%%%%%%%%%%%%%%%%%%%%%%%%%%%%%%%%%%%%%%%
For $I \in \{0,1\}^n$, we put $M^I = M_1^{i_1} M_2^{i_2} \cdots M_n^{i_n}$. 
By using these notations, we have 
$$
M^I \ee_{j_1,\dots,j_n} 
= (G_1^{i_1} \ee_{j_1}) \otimes \cdots \otimes (G_n^{i_n} \ee_{j_n}). 
$$
We often use the vectors 
$$
 \ff_I = M^I \ee_{1,\dots,1} 
= (G_1^{i_1} \ee_1) \otimes \cdots \otimes (G_n^{i_n} \ee_1) ,\quad 
I=(i_1,\ldots,i_n) \in \{0,1\}^n. 
$$
Let $\Mon$ be the monodromy group generated by $M_0,M_1,\ldots,M_n$, 
which is a subgroup of $\GL_{2^n}(\CC)\simeq \GL((\CC^2)^{\otimes n})$. 
%%%%%%%%%%%%%%%%%%%%%%%%%%%%%%%%%%%%%%%%%%%%%%%%%%%%%%%%%%%%%%%%%%%%%%%%%%%%%%%%%%%%%%%
%% Proposition %%%%%%%%%%%%%%%%%%%%%%%%%%%%%%%%%%%%%%%%%%%%%%%%%%%%%%%%%%%%%%%%%%%%%%
\begin{prop}[(\cite{HT}, \cite{GM2})] \label{prop-irr}
  We assume 
  \begin{align*}
    ({\rm irr-abc}) \qquad \qquad 
    a-\sum_{k=1}^n i_k c_k ,\quad b-\sum_{k=1}^n i_k c_k \not\in \ZZ ,\qquad 
    \forall I=(i_1,\ldots,i_n)\in \{0,1\}^n, 
  \end{align*}
  or equivalently, 
  \begin{align*}
    ({\rm irr-}\alpha \beta \gamma) \qquad \qquad 
    \alpha-\prod_{k=1}^n \gamma_k^{i_k} ,\quad 
    \beta-\prod_{k=1}^n \gamma_k^{i_k} \neq 0 ,\qquad 
    \forall I=(i_1,\ldots,i_n) \in \{0,1\}^n.
  \end{align*}
Then we have
\\
{\rm (1)} \ The vectors $\ff_I$ $(I=(i_1,\ldots,i_n) \in \{0,1\}^n)$ form 
    a basis of $\CC^{2^n}\cong (\CC^2)^{\otimes n}$.
\\
{\rm (2)} \ The monodromy group $\Mon$ acts on $\CC^{2^n}$ irreducibly. 
\end{prop}
%%%%%%%%%%%%%%%%%%%%%%%%%%%%%%%%%%%%%%%%%%%%%%%%%%%%%%%%%%%%%%%%%%%%%%%%%%%%%%%%%%%%%%%%
We consider the case of $a=b=\frac{1}{2}$, $c_k=1$ (i.e., $\alpha=\beta=-1$, $\gamma_k=1$) 
in detail.
%%%%%%%%%%%%%%%%%%%%%%%%%%%%%%%%%%%%%%%%%%%%%%%%%%%%%%%%%%%%%%%%%%%%%%%%%%%%%%%%%%%%%%%%%
%% Corollary %%%%%%%%%%%%%%%%%%%%%%%%%%%%%%%%%%%%%%%%%%%%%%%%%%%%%%%%%%%%%%%%%%%%%%%%%%%%
\begin{cor} \label{cor-rep-matrix}
{\rm (1)} Assume that $a, b, c_k \in \RR$. Since ${}^{\vee}$ means the complex conjugation, 
the intersection matrix $H$ is a Hermitian matrix if $n$ is even and a skew Hermitian 
matrix if $n$ is odd. Further assume that 
$a + b \in \ZZ, \ c_k \in \frac{1}{2} \ZZ$ and $\sum_{k=1}^n c_k \in \ZZ$
{\rm (}that is, $\alpha = \overline{\beta}, \ \gamma_k = \pm 1$ and 
$\prod_{k=1}^n \gamma_k = 1$ {\rm )}. Then $M_k$ and $H$ are defined over $\RR$. In this case, 
the monodromy group $\Mon$ is a subgroup of a real orthogonal/symplectic group 
with respect to $H$. 
\\
{\rm (2)}
  In the case of $a=b=\frac{1}{2}$, $c_k=1$, the representation matrices are as follows. 
  For $k=1,\dots, n$, we have 
  \begin{align*}
    M_k =E_2 \otimes \cdots \otimes E_2 \otimes \underset{k\textrm{-th}}{G_k}
    \otimes E_2 \otimes \cdots \otimes E_2 ,
    \quad 
    G_k =\left(
      \begin{array}{cc}
        1 & -1 \\ 0 & 1
      \end{array}
      \right) .
  \end{align*}
  The $(I,I')$-entry of $M_0$ is 
  \begin{align*}
    \left\{ 
      \begin{array}{ll}
        1 & (I=I' \neq (1,\dots,1)), \\
        (-1)^{n+1} & (I=I'=(1,\dots,1)) ,\\
        (-1)^{n+1} \cdot 4& (I=(1,\dots,1), I'=(0,\dots ,0)) ,\\
        (-1)^{n+1}\cdot 2 & (I=(1,\dots,1),I'\neq(0,\dots ,0),(1,\dots,1) 
                             \textrm{ and } |I'|\equiv 0\mod 2) \\
        0 & (\textrm{otherwise}).
      \end{array}
    \right.
  \end{align*}
  We have ${}^t M_i \cdot H \cdot M_i =H$ and ${}^t H =(-1)^n H$.  
  Each entry of $H$ belongs to $\ZZ[\frac{1}{2}]$. 
\end{cor}
%%%%%%%%%%%%%%%%%%%%%%%%%%%%%%%%%%%%%%%%%%%%%%%%%%%%%%%%%%%%%%%%%%%%%%%%%%%%%%%%%%
\subsection{Zariski closure and Reflection group}
Let us consider the Zariski closure of the monodromy group $\Mon$ 
after Beukers and Heckman. As in \cite{BH}, we call a linear map $g \in \GL_n(\CC)$ 
a {\it reflection} if $g - \mathrm{Id}$ has rank one 
(Hence a reflection may be of infinite order, and our reflections include matrices 
called {\it transvection}). We call the determinant of a reflection $g$ the 
{\it special eigenvalue} of $g$. 
For a subgroup $G \subset \GL_n(\CC)$, the Zariski closure of $G$ over complex numbers 
is denoted by $\overline{G}$. The connected component (in the Zariski topology) of the 
identity, which is a normal subgroup, is denoted by $G^0$. The quotient group 
$G/G^0 \cong \overline{G} / \overline{G}^0$ is a finite group. 
We apply the following Proposition with $r = M_0$ (see Remark \ref{rem-M0}).   
%%% Beukers and Heckman %%%%%%%%%%%%%%%%%%%%%%%%%%%%%%%%%%%%%%%%%%%%%%%%%%%%%%%%%%%%
\begin{prop}[(\cite{BH})] \label{BH-prop}
Suppose $G \subset \SL_n(\CC)$ is a connected algebraic group 
acting irreducibly on $\CC^n$. Let $r \in \GL_n(\CC)$ be a reflection with special 
eigenvalue $\delta \in \CC^{\times}$ which normalizes $G$. Then we have the following 
three possibilities,
\\
{\rm (I)} If $\delta \ne \pm 1$ then $G = \SL_n(\CC)$,
\\
{\rm (II)} If $\delta = +1$ then $G = \SL_n(\CC)$ or $G = \Sp_n(\CC)$,
\\
{\rm (III)} If $\delta = -1$ then $G = \SL_n(\CC)$ or $G = \SO_n(\CC)$.
\end{prop}
%%%%%%%%%%%%%%%%%%%%%%%%%%%%%%%%%%%%%%%%%%%%%%%%%%%%%%%%%%%%%%%%%%%%%%%%%%%%%%%%%%
An immediate consequence is
%%% Theorem %%%%%%%%%%%%%%%%%%%%%%%%%%%%%%%%%%%%%%%%%%%%%%%%%%%%%%%%%%%%%%%%%%%%%%
\begin{tm} \label{main}
Assume that $\Mon^0$ acts on $\CC^{2^m}$ irreducibly. Then we have the following 
three possibilities,
\\
{\rm (I)} If $\delta_0 \ne \pm 1$ then $\SL_{2^n}(\CC) \subset \overline{\Mon}$,
\\
{\rm (II)} If $\delta_0 = +1$ then $\SL_{2^n}(\CC) \subset \overline{\Mon}$ or 
$\Sp_{2^n}(\CC) \subset \overline{\Mon} \subset \mathbf{GSp}_{2^n}(\CC)$,
\\
{\rm (III)} If $\delta_0 = -1$ then $\SL_{2^n}(\CC) \subset \overline{\Mon}$ or 
$\SO_{2^n}(\CC) \subset \overline{\Mon} \subset \mathbf{GO}_{2^n}(\CC)$.
\\
Moreover we have $\overline{\Mon}^0 \subset \SL_{2^n}(\CC)$ if $a,b,c_i \in \QQ$.
\end{tm}
%%% proof %%%%%%%%%%%%%%%%%%%%%%%%%%%%%%%%%%%%%%%%%%%%%%%%%%%%%%%%%%%%%%%%%%%%%%%%%
\begin{proof}
Let $\mathfrak{m}$ be the Lie algebra of $\overline{\Mon}$.
By the assumption, $\mathfrak{m}$ acts on $\CC^{2^n}$ irreducibly, and so is
$\mathfrak{m} \cap \mathfrak{sl}_{2^n}(\CC)$. Therefore 
$G = (\overline{\Mon} \cap \SL_{2^n}(\CC))^0$ ats on $\CC^{2^n}$ irreducibly. Since  
$G$ is normalized by $M_0$, we can apply Proposition \ref{BH-prop}. We have the above 
three cases, since $\overline{\Mon}$ normalizes $G$ and the normalizer of $\Sp$ and $\SO$ 
in $\GL$ are $\mathbf{GSp}$ and $\mathbf{GO}$ respectively.  
If $a, b, c_k \in \QQ$, then the image of 
$\det : \Mon \rightarrow \CC^{\times}$ is finite and we have $\Mon^0 \subset \SL_{2^n}(\CC)$. 
\end{proof} 
%%% Remark %%%%%%%%%%%%%%%%%%%%%%%%%%%%%%%%%%%%%%%%%%%%%%%%%%%%%%%%%%%%%%%%%%%%%%%%%
\begin{rem}
 For $n=2$, it was shown by Sasaki that $\overline{\Mon} = \GL_4(\CC)$ if parameters 
are general complex numbers in {\rm \cite{Sa}}. The same is true for $n \geq 3$ by 
{\rm (I)} of the above theorem {\rm (}see also Corollary {\rm \ref{main3-cor})}. 
\end{rem}
%%% Corollary %%%%%%%%%%%%%%%%%%%%%%%%%%%%%%%%%%%%%%%%%%%%%%%%%%%%%%%%%%%%%%%%%%%%%%
\begin{cor} \label{cor-main}
Assume that $\Mon^0$ acts on $\CC^{2^m}$ irreducibly, and that 
$a, b \in \RR, \ a +b \in \ZZ, \ c_k \in \frac{1}{2} \ZZ$ and $\sum_{k=1}^n c_k \in \ZZ$. 
Then we have 
$\overline{\Mon} = \mathbf{O}_{2^n}(\CC)$ if $n$ is even, and 
$\overline{\Mon} = \Sp_{2^n}(\CC)$ if $n$ is odd.
\end{cor}
%%% proof %%%%%%%%%%%%%%%%%%%%%%%%%%%%%%%%%%%%%%%%%%%%%%%%%%%%%%%%%%%%%%%%%%%%%%%%%
\begin{proof}
By Corollary \ref{cor-rep-matrix}, we have 
$\Mon \subset \mathbf{O}_{2^n}(\CC)$ if $n$ is even, and 
$\Mon \subset \Sp_{2^n}(\CC)$ if $n$ is odd. 
Since $\delta_0 = (-1)^{n+1}$, the assertion follows from Theorem \ref{main}.
\end{proof}
%%%%%%%%%%%%%%%%%%%%%%%%%%%%%%%%%%%%%%%%%%%%%%%%%%%%%%%%%%%%%%%%%%%%%%%%%%%%%%%%%%%
Because of the above theorem, it is interesting to determine conditions for 
irreducibility of  $\Mon^0$. We give a partial answer to this problem in the 
following (Corollary \ref{main3-cor}). 
Let $\Ref \subset \Mon$ be the smallest normal subgroup containing $M_0$, that is, 
a subgroup generated by reflections $gM_0g^{-1} \ (g \in \Mon)$. 
Since $\Ref^0 \subset \Mon^0$, the irreducibility of $\Ref^0$ is a sufficient condition 
for irreducibility of  $\Mon^0$.
The reflection subgroup was introduced in \cite{BH} for the generalized hypergeometric 
function $_{n}F_{n-1}$, and considered in \cite{Ka} for Appell's $F_4$ to study the 
finiteness condition. 
%%%% Proposition %%%%%%%%%%%%%%%%%%%%%%%%%%%%%%%%%%%%%%%%%%%%%%%%%%%%%%%%%%%%%%%%%%
\begin{prop} \label{finite-prop}
The monodromy group $\Mon$ is finite if and only if $\Ref$ is finite.
\end{prop}
%%% proof %%%%%%%%%%%%%%%%%%%%%%%%%%%%%%%%%%%%%%%%%%%%%%%%%%%%%%%%%%%%%%%%%%%%%%%%%
\begin{proof}    
Let us assume that $\Ref$ is a finite group. Since 
$\{ M_1^d M_0 M_1^{-d} \ | \ d =1,2, \dots\}$ is a finite set, there exist $k$ and $l$
($k \ne l$) such that $M_1^k M_0 M_1^{-k} = M_1^l M_0 M_1^{-l}$, namely,   
$M_1^{k-l}M_0 = M_0 M_1^{k-l}$. On the other hand, we have
\[
 M_1^{-d} = G_1^{-d} \otimes E_2 \otimes \cdots \otimes E_2, \qquad
 G_1^{-d} = \begin{pmatrix} 1 & \frac{1-\gamma_1^d}{1-\gamma_1} \\ 
0 & \gamma_1^d \end{pmatrix}
\]
for $d=1,2,\dots$, and
\begin{align*}
  M_1^{-d} M_0 \ee_{1,\dots,1} &= \delta_0 M_1^{-d} \ee_{1,\dots,1} 
= \delta_0 (\frac{1-\gamma_1^d}{1-\gamma_1} \ee_{0,1,\dots,1} + \gamma_1^d \ee_{1,\dots,1}),
\\
 M_0 M_1^{-d}\ee_{1,\dots,1} 
&= M_0(\frac{1-\gamma_1^d}{1-\gamma_1} \ee_{0,1,\dots,1} + \gamma_1^d \ee_{1,\dots,1})
= \frac{1-\gamma_1^d}{1-\gamma_1} M_0 \ee_{0,1,\dots,1} 
+ \delta_0 \gamma_1^d \ee_{1,\dots,1}.
\end{align*}
Therefore, if $M_1^d$ commutes with $M_0$, we have
\begin{align*}
  \mathbf{0} = (M_0 M_1^{-d} - M_1^{-d} M_0)\ee_{1,\dots,1} 
&= \frac{1-\gamma_1^d}{1-\gamma_1} (M_0 - \delta_0 \mathrm{Id}) \ee_{0,1,\dots,1} \\
&= \frac{1-\gamma_1^d}{1-\gamma_1} \big( (1-\delta_0)\ee_{0,1,\dots,1} 
+ (\gamma_1 + \delta_0)\ee_{1,\dots,1}
\big).
\end{align*}
If $\gamma_1 \ne -1$, this implies $\gamma_1^d = 1$ and $M_1^d = \mathrm{Id}$. If 
$\gamma_1 = -1$, we have $G_1=\begin{pmatrix} 1 & 1 \\ 0 & -1 \end{pmatrix}$ 
and $M_1^2 = \mathrm{Id}$. By the similar argument, we see that $M_k \ (k=1,\dots,n)$ 
are of finite order. Therefore $\Mon / \Ref$ is finite, and so is $\Mon$. 
The converse is obvious. 
\end{proof}
%%%%%%%%%%%%%%%%%%%%%%%%%%%%%%%%%%%%%%%%%%%%%%%%%%%%%%%%%%%%%%%%%%%%%%%%%%%%%%%%
 Let us consider reflections
 \begin{align*}
   R_I = M^I M_0 (M^I)^{-1} = E_{2^n} - M^I N_0 (M^I)^{-1} \qquad I \in \{0,1\}^n
 \end{align*}
and endomorphisms $N_I = E_{2^n} - R_I = M^I N_0 (M^I)^{-1}$ of $\CC^{2^n}$. 
%%% Lemma %%%%%%%%%%%%%%%%%%%%%%%%%%%%%%%%%%%%%%%%%%%%%%%%%%%%%%%%%%%%%%%%%%%%%%%%%
\begin{lem} \label{lem-NI}
\rm{(1)} \ The image of $N_I$ is spanned by $\ff_I$, and we have
\[
\ker N_I = M^I \cdot \ker N_0 = \{ w \in \CC^{2^n} \ | \ {}^tw H \ff_I^{\vee} = 0 \}.
\]
\rm{(2)} \ A linear map 
\[
\nu \ : \ \CC^{2^n} \rightarrow \oplus_{I \in \{0,1\}^n} \CC \ff_I, 
\quad w \mapsto (\dots, N_I w, \dots)
\] 
is an isomorphism under the condition $(\mathrm{irr - \alpha \beta \gamma})$. 
\end{lem}
%%% proof %%%%%%%%%%%%%%%%%%%%%%%%%%%%%%%%%%%%%%%%%%%%%%%%%%%%%%%%%%%%%%%%%%%%%%%%%
\begin{proof}
  (1) We have
\[
N_I \cdot \CC^{2^n} = M^I N_0 (M^I)^{-1} \cdot \CC^{2^n}
= M^I N_0 \cdot \CC^{2^n} = M^I \cdot \CC \ee_{1,\dots,1} = \CC \ff_I 
\]
and 
\begin{align*}
w \in \ker N_I \ \Leftrightarrow \ M^I N_0 (M^I)^{-1}w = \mathbf{0}
\ &\Leftrightarrow \ N_0 (M^I)^{-1}w = \mathbf{0} \\
&\Leftrightarrow \ (M^I)^{-1}w \in \ker N_0 \\
&\Leftrightarrow \ {}^t((M^I)^{-1}w) H \ee_{1,\dots,1} = 0 \\
&\Leftrightarrow \ {}^tw H ((M^I)^{\vee} \ee_{1,\dots,1}) = 0
\ \Leftrightarrow \ {}^tw H \ff_I^{\vee} = 0.
\end{align*}
(2) \ Since $\ff_I$ are linearly independent under the condition 
($\mathrm{irr - \alpha \beta \gamma}$), we see that $\ff_I^{\vee}$ are 
linearly independent and $\ker \nu = \mathbf{0}$.  
\end{proof}
%%%%%%%%%%%%%%%%%%%%%%%%%%%%%%%%%%%%%%%%%%%%%%%%%%%%%%%%%%%%%%%%%%%%%%%%%%%%%%%%%%
%%% theorem %%%%%%%%%%%%%%%%%%%%%%%%%%%%%%%%%%%%%%%%%%%%%%%%%%%%%%%%%%%%%%%%%%%%%%
\begin{tm} \label{main2}
Assume that $\Mon$ is irreducible {\rm(} that is, $(\mathrm{irr - \alpha \beta \gamma})$
holds {\rm)}. The action of $\Ref$ is irreducible if and only if 
  at most one of $\gamma_1 ,\ldots ,\gamma_n ,\alpha \beta^{-1}$ is $-1$. 
\end{tm}
%%%%%%%%%%%%%%%%%%%%%%%%%%%%%%%%%%%%%%%%%%%%%%%%%%%%%%%%%%%%%%%%%%%%%%%%%%%%%%%%%%%
We divide the proof into the following four Lemmas, where we always assume that 
{\it the action of $\Mon$ is irreducible}. 
%%% Lemma %%%%%%%%%%%%%%%%%%%%%%%%%%%%%%%%%%%%%%%%%%%%%%%%%%%%%%%%%%%%%%%%%%%%%%%%%
\begin{lem} \label{Ref-lem1}
If none of $\gamma_1, \dots, \gamma_n$ is $-1$, then the action of $\Ref$ 
is irreducible. 
\end{lem}
%%% proof %%%%%%%%%%%%%%%%%%%%%%%%%%%%%%%%%%%%%%%%%%%%%%%%%%%%%%%%%%%%%%%%%%%%%%%%%
\begin{proof}
First note that if $\Ref$ acts on a subspace $W$, then $W$ must be a direct sum of 
1-dimensional subspaces $\CC \ff_I$ since $N_I$ acts on $W$ as an endomorphism.
Let $W \ne \mathbf{0}$ be an irreducible $\Ref$-subspace.   
Since $\Ref$ is a normal subgroup, $gW$ is also irreducible $\Ref$-subspace for 
$g \in \Mon$. Replacing $W$ by $M^IW$ with $I \in \ZZ^n$ if necessary, we assume that 
$\ee_{1,\dots,1} = \ff_{0,\dots,0} \in W$. Then we have $M_1 \ee_{1,\dots,1} \in M_1W$, and 
hence 
\begin{align*}
 N_0 (M_1 \ee_{1,\dots,1}) &= -\frac{1}{\gamma_1} N_0 (\ee_{0,1,\dots,1} - \ee_{1,\dots,1}) 
\\
&= -\frac{1}{\gamma_1}
\Big( (-\gamma_1 + (-1)^n \frac{\gamma_1 \cdots \gamma_n}{\alpha \beta})\ee_{1,\dots,1}
-(1 + (-1)^n \frac{\gamma_1 \cdots \gamma_n}{\alpha \beta})\ee_{1,\dots,1} \Big)
\\
&= \frac{1 + \gamma_1}{\gamma_1} \ee_{1,\dots,1},
\end{align*}
belongs to $M_1W$. By the assumption, this is not $\mathbf{0}$.
Therefore we have $\ee_{1,\dots,1} \in W \cap M_1W$ and $M_1W = W$ by the irreducibility. 
We see that every $M_k$ acts on $W$ in the same way. Hence we have
 $W = \CC^{2^n}$.  
\end{proof}
%%%%%%%%%%%%%%%%%%%%%%%%%%%%%%%%%%%%%%%%%%%%%%%%%%%%%%%%%%%%%%%%%%%%%%%%%%%%%%%%%%
%%%%%%%%%%%%%%%%%%%%%%%%%%%%%%%%%%%%%%%%%%%%%%%%%%%%%%%%%%%%%%%%%%%%%%%%%%%%%%%%%%
\begin{lem}
  If exactly one of $\gamma_1 ,\ldots ,\gamma_n$ is $-1$ and 
  $\alpha \beta^{-1}$ is not equal to $-1$, 
  then the action of $\Ref$ is irreducible.   
\end{lem}
\begin{proof}
  For simplicity, we may assume that 
  $$
  \gamma_1=-1,\quad \gamma_k\neq -1 \ (k=2,\ldots ,n) ,\quad \beta\neq -\alpha .
  $$
  Let $W \ne \mathbf{0}$ be an irreducible $\Ref$-subspace 
  such that $\ee_{1,\dots,1} \in W$ as in the proof of Lemma \ref{Ref-lem1}. 
  By the assumption $\gamma_k\neq -1 \ (k=2,\ldots ,n)$ and the proof of  
  Lemma \ref{Ref-lem1}, we have $M_2^{i_2} \cdots M_n^{i_n} \ee_{1,\dots,1}\in W$. 
  Since 
  \begin{align*}
    W \ni M_2 \ee_{1,\dots,1} 
    &=\ee_1 \otimes (G_2 \ee_1) \otimes \cdots \otimes \ee_1
    =\ee_1 \otimes (-\gamma_2^{-1}\ee_0 +\gamma_2^{-1} \ee_1) \otimes \cdots \otimes \ee_1 \\
    &=-\gamma_2^{-1}\ee_{1,0,1,1,\dots ,1}+\gamma_2^{-1} \ee_{1,\dots, 1} ,
  \end{align*}
  and $\ee_{1,\dots ,1} \in W$, we obtain $\ee_{1,0,1,1,\dots ,1} \in W$. 
  Similar arguments show that $\ee_{1,i_2 ,\dots ,i_n} \in W$ for any 
  $(i_2,\dots ,i_n)\in \{0,1\}^{n-1}$. 
  In particular, we have $\ee_{1,0,\dots ,0}\in W$.
  Note that the condition $\gamma_1 =-1$ 
  implies $G_1=\begin{pmatrix} 1 & 1 \\ 0 & -1 \end{pmatrix}$ and  
  \begin{align*}
    M_1 \ee_{0,i_2,\dots ,i_n} =\ee_{0,i_2,\dots ,i_n} ,\quad
    M_1 \ee_{1,i_2,\dots ,i_n} =\ee_{0,i_2,\dots ,i_n}-\ee_{1,i_2,\dots ,i_n} .
  \end{align*}
  By Proposition \ref{prop-rep-matrix}, we have 
  \begin{align*}
    &M_0 M_1 \ee_{1,0,\dots ,0} =M_0(\ee_{0,0,\dots ,0}-\ee_{1,0,\dots ,0}) \\
    &=\left( \ee_{0,0,\dots ,0}
      +(-1)^n \frac{(\alpha-1)(\beta-1)\prod_{k=2}^n \gamma_k}{\alpha \beta}\ee_{1,\dots, 1} \right) 
    -\left( \ee_{1,0,\dots ,0} -(-1)^{n+1} 
      \frac{(\alpha \beta +1)\prod_{k=2}^n \gamma_k}{\alpha \beta}\ee_{1,\dots, 1} \right) \\
    &=\ee_{0,0,\dots ,0}-\ee_{1,0,\dots ,0} +\lambda \ee_{1,\dots, 1} ,
  \end{align*}
  where we put 
  $$ 
  \lambda =
  (-1)^n \frac{\prod_{k=2}^n \gamma_k}{\alpha \beta}\Big( (\alpha-1)(\beta-1)-(\alpha \beta +1) \Big)
  =(-1)^n \frac{\prod_{k=2}^n \gamma_k}{\alpha \beta} (-\alpha -\beta) . 
  $$
  By the condition $\beta \neq -\alpha$, we have $\lambda \neq 0$. 
  Because of 
  \begin{align*}
    W \ni M_1 M_0 M_1^{-1}\ee_{1,0,\dots ,0} 
    &= M_1 M_0 M_1 \ee_{1,0,\dots ,0}\\
    &=M_1 (\ee_{0,0,\dots ,0}-\ee_{1,0,\dots ,0} +\lambda \ee_{1,\dots, 1}) \\
    &=\ee_{1,0,\dots ,0} +\lambda \ee_{0,1,\dots ,1}-\lambda \ee_{1,\dots ,1}, 
  \end{align*}
  $\ee_{1,0,\dots ,0} ,\ee_{1,\dots ,1} \in W$ and $\lambda \neq 0$, 
  we obtain $\ee_{0,1,\dots ,1} \in W$.
  By using $M_2^{i_2} \cdots M_n^{i_n} \ee_{0,1,\dots ,1}\in W$, 
  we can show that $\ee_{i_1,i_2 ,\dots ,i_n} \in W$ for all $(i_1,\dots ,i_n)\in \{0,1\}^{n}$. 
  Hence we obtain $W=\CC^{2^n}$. 
\end{proof}

\begin{lem}\label{lem-red-1}
  If at least two of $\gamma_1 ,\ldots ,\gamma_n$ are $-1$, 
  then the action of $\Ref$ is reducible. 
\end{lem}
\begin{lem}\label{lem-red-2}
  If at least one of $\gamma_1 ,\ldots ,\gamma_n$ is $-1$ and 
  $\alpha \beta^{-1}$ is $-1$, 
  then the action of $\Ref$ is reducible. 
\end{lem}
We show these lemmas by applying ideas given in the proofs of \cite[Lemmas 4.2 and 4.3]{Ka}. 
To prove the lemmas, we use relations 
\begin{align}
  \label{eq-G1}
  \begin{pmatrix} 1 & 1 \\ 0 & -1 \end{pmatrix} \ee_0 =\ee_0 ,\quad 
  \begin{pmatrix} 1 & 1 \\ 0 & -1 \end{pmatrix} \ee_1 =\ee_0-\ee_1 ,\quad 
  \begin{pmatrix} 1 & 1 \\ 0 & -1 \end{pmatrix} (2\ee_1 -\ee_0) =-(2\ee_1 -\ee_0) .
\end{align}
\begin{proof}[Proof of Lemma \ref{lem-red-1}]  
  For simplicity, we may assume that $\gamma_1 =\gamma_2 =-1$. 
  For each $(i_3 ,\dots ,i_n)$, we put  
  \begin{align*}
    g_{0,0,i_3,\dots ,i_n} &= \ee_{0,0,i_3,\dots ,i_n} 
    =\ee_0 \otimes \ee_0 \otimes \ee_{i_3} \otimes \cdots \otimes \ee_{i_n} , \\
    g_{1,0,i_3,\dots ,i_n} &= 2\ee_{1,0,i_3,\dots ,i_n}- \ee_{0,0,i_3,\dots ,i_n} 
    =(2\ee_1 -\ee_0) \otimes \ee_0 \otimes \ee_{i_3} \otimes \cdots \otimes \ee_{i_n} , \\
    g_{0,1,i_3,\dots ,i_n} &= 2\ee_{0,1,i_3,\dots ,i_n}- \ee_{0,0,i_3,\dots ,i_n} 
    =\ee_0 \otimes (2\ee_1 -\ee_0) \otimes \ee_{i_3} \otimes \cdots \otimes \ee_{i_n} , \\
    g_{1,1,i_3,\dots ,i_n} &= 4\ee_{1,1,i_3,\dots ,i_n} -2\ee_{1,0,i_3,\dots ,i_n}
    -2\ee_{0,1,i_3,\dots ,i_n}+ \ee_{0,0,i_3,\dots ,i_n} 
    =(2\ee_1 -\ee_0) \otimes (2\ee_1 -\ee_0) \otimes \ee_{i_3} \otimes \cdots \otimes \ee_{i_n} ,\\
    f_{14;i_3 ,\dots ,i_n}^{\pm} &=g_{0,0,i_3,\dots ,i_n} \pm g_{1,1,i_3,\dots ,i_n} , \qquad 
    f_{23;i_3 ,\dots ,i_n}^{\pm} =g_{1,0,i_3,\dots ,i_n} \pm g_{0,1,i_3,\dots ,i_n} .
  \end{align*}
  We consider proper subspaces 
  \begin{align*}
    W^{\pm} = \bigoplus_{(i_3 ,\dots ,i_n)} \CC f_{14;i_3 ,\dots ,i_n}^{\pm}
    \oplus \bigoplus_{(i_3 ,\dots ,i_n)} \CC f_{23;i_3 ,\dots ,i_n}^{\pm}
  \end{align*}
  of $\CC^{2^n}$ whose dimensions are $2\cdot 2^{n-2}=2^{n-1}$. 
  We show that these are non-trivial $\Ref$-subspaces. 
  Note that by the definition, we have 
  $$
  \ee_{1,\dots ,1}
  =\frac{1}{4} (g_{0,0,1,\dots ,1}+g_{1,0,1,\dots ,1}+g_{0,1,1,\dots ,1}+g_{1,1,1,\dots ,1})
  =\frac{1}{4} (f_{14;1,\dots ,1}^{+}+f_{23;1,\dots ,1}^{+} ) \in W^{+} .
  $$

  By (\ref{eq-G1}), actions of $M_1$ and $M_2$ are given as 
  \begin{align*}
    M_1 : \left\{
      \begin{array}{l}
        g_{0,i_2,i_3,\dots ,i_n} \mapsto g_{0,i_2,i_3,\dots ,i_n} \\
        g_{1,i_2,i_3,\dots ,i_n} \mapsto -g_{1,i_2,i_3,\dots ,i_n}
      \end{array}
    \right. ,\qquad 
    M_2 : \left\{
      \begin{array}{l}
        g_{i_1,0,i_3,\dots ,i_n} \mapsto g_{i_1,0,i_3,\dots ,i_n} \\
        g_{i_1,1,i_3,\dots ,i_n} \mapsto -g_{i_1,1,i_3,\dots ,i_n} ,
      \end{array}
    \right. 
  \end{align*}
  and hence 
  \begin{align*}
    M_1 : \left\{
      \begin{array}{l}
        f_{14;i_3 ,\dots ,i_n}^{\pm} \mapsto f_{14;i_3 ,\dots ,i_n}^{\mp} \\
        f_{23;i_3 ,\dots ,i_n}^{\pm} \mapsto -f_{23;i_3 ,\dots ,i_n}^{\mp}
      \end{array}
    \right. ,\qquad 
    M_2 : \left\{
      \begin{array}{l}
        f_{14;i_3 ,\dots ,i_n}^{\pm} \mapsto f_{14;i_3 ,\dots ,i_n}^{\mp} \\
        f_{23;i_3 ,\dots ,i_n}^{\pm} \mapsto f_{23;i_3 ,\dots ,i_n}^{\mp} .
      \end{array}
    \right. 
  \end{align*}
  These imply that 
  $M_1 W^{\pm} =W^{\mp}$,
  $M_2 W^{\pm} =W^{\mp}$. 
    
  We can show that 
  $M_i W^{\pm} =W^{\pm}$ $(i=3,\dots ,n)$. 
  For example, if $i=3$, we have 
  \begin{align*}
    M_3 \cdot g_{i_1,i_2,0,i_4,\dots ,i_n}
    =g_{i_1,i_2,0,i_4,\dots ,i_n}, \quad 
    M_3 \cdot g_{i_1,i_2,1,i_4,\dots ,i_n}
    =-\gamma_3^{-1} g_{i_1,i_2,0,i_4,\dots ,i_n} +\gamma_3^{-1} g_{i_1,i_2,1,i_4,\dots ,i_n}
  \end{align*}
  and 
  \begin{align*}
    &M_3 \cdot f_{14; 0,i_4,\dots ,i_n}^{\pm}
    =f_{14; 0,i_4,\dots ,i_n}^{\pm}, \quad 
    M_3 \cdot f_{14; 1,i_4,\dots ,i_n}^{\pm}
    =-\gamma_3^{-1} f_{14; 0,i_4,\dots ,i_n}^{\pm} +\gamma_3^{-1} f_{14; 1,i_4,\dots ,i_n}^{\pm} ,\\
    &M_3 \cdot f_{23; 0,i_4,\dots ,i_n}^{\pm}
    =f_{23; 0,i_4,\dots ,i_n}^{\pm}, \quad 
    M_3 \cdot f_{23; 1,i_4,\dots ,i_n}^{\pm}
    =-\gamma_3^{-1} f_{23; 0,i_4,\dots ,i_n}^{\pm} +\gamma_3^{-1} f_{23; 1,i_4,\dots ,i_n}^{\pm} , 
  \end{align*}
  which imply $M_3 W^{\pm} =W^{\pm}$. 
  
  By Proposition \ref{prop-rep-matrix} and $\gamma_1 =\gamma_2 =-1$, we have  
  \begin{align*}
    M_0 \ee_{0,0,i_3 ,\dots ,i_n} 
    &=\ee_{0,0,i_3 ,\dots ,i_n}-\lambda_{0;i_3 ,\dots ,i_n}\ee_{1,\dots ,1}  ,\\
    M_0 \ee_{1,0,i_3 ,\dots ,i_n} 
    &=\ee_{1,0,i_3 ,\dots ,i_n}-\lambda_{1;i_3 ,\dots ,i_n}\ee_{1,\dots ,1}, \\
    M_0 \ee_{0,1,i_3 ,\dots ,i_n} 
    &=\ee_{0,1,i_3 ,\dots ,i_n}-\lambda_{1;i_3 ,\dots ,i_n}\ee_{1,\dots ,1}, \\
    M_0 \ee_{1,1,i_3 ,\dots ,i_n} 
    &=\ee_{1,1,i_3 ,\dots ,i_n}-\lambda_{1;i_3 ,\dots ,i_n}\ee_{1,\dots ,1} ,
  \end{align*}
  where 
  \begin{align*}
    &\lambda_{1;i_3 ,\dots ,i_n} =(-1)^{n+i_3+\dots +i_n}
    \frac{(\alpha \beta +(-1)^{i_3+\dots +i_n} 
      \prod_{k=3}^n \gamma_k^{i_k})\prod_{k=3}^n \gamma_k^{1-i_k}}{\alpha \beta} ,\\
    &\lambda_{0;i_3 ,\dots ,i_n} =\left\{ 
      \begin{array}{ll}
        (-1)^n \frac{(\alpha-1)(\beta-1)\prod_{k=3}^n \gamma_k}{\alpha \beta}
        & ((i_3 ,\dots ,i_n)=(0,\dots ,0)) \\
        \lambda_{1;i_3 ,\dots ,i_n}
        & ((i_3 ,\dots ,i_n)\neq (0,\dots ,0)) .
      \end{array}
    \right. 
  \end{align*}
  Thus we obtain 
  \begin{align*}
    M_0 \cdot g_{0,0,i_3,\dots ,i_n} 
    &=g_{0,0,i_3 ,\dots ,i_n}-\lambda_{0;i_3 ,\dots ,i_n}\ee_{1,\dots ,1}  , \\
    M_0 \cdot g_{1,0,i_3,\dots ,i_n} 
    &=g_{1,0,i_3,\dots ,i_n}-(2\lambda_{1;i_3 ,\dots ,i_n} -\lambda_{0;i_3 ,\dots ,i_n})\ee_{1,\dots ,1} , \\
    M_0 \cdot g_{0,1,i_3,\dots ,i_n} 
    &=g_{0,1,i_3,\dots ,i_n} -(2\lambda_{1;i_3 ,\dots ,i_n} -\lambda_{0;i_3 ,\dots ,i_n})\ee_{1,\dots ,1} , \\
    M_0 \cdot g_{1,1,i_3,\dots ,i_n} 
    &=g_{1,1,i_3,\dots ,i_n}-\lambda_{0;i_3 ,\dots ,i_n}\ee_{1,\dots ,1} , 
  \end{align*}
  and 
  \begin{align*}
    M_0 \cdot f_{14;i_3,\dots ,i_n}^{-}&=f_{14;i_3,\dots ,i_n}^{-}\in W^{-} , \quad 
    M_0 \cdot f_{23;i_3,\dots ,i_n}^{-}=f_{23;i_3,\dots ,i_n}^{-}\in W^{-} , \\
    M_0 \cdot f_{14;i_3,\dots ,i_n}^{+}&=f_{14;i_3,\dots ,i_n}^{+}
    -2\lambda_{0;i_3 ,\dots ,i_n}\ee_{1,\dots ,1} \in W^{+}, \\ 
    M_0 \cdot f_{23;i_3,\dots ,i_n}^{+}&=f_{23;i_3,\dots ,i_n}^{+}
    -2(2\lambda_{1;i_3 ,\dots ,i_n} -\lambda_{0;i_3 ,\dots ,i_n})\ee_{1,\dots ,1} \in W^{+} .
  \end{align*}
  These mean $M_0 W^{\pm} =W^{\pm}$. 

  Now we show that $W^{\pm}$ are $\Ref$-subspaces. 
  To prove this claim, it is sufficient to see that $(gM_0 g^{-1})W^{\pm} =W^{\pm}$, 
  for each generator $gM_0 g^{-1}$ ($g\in \Mon$) of $\Ref$. 
  Since $g$ is represented as a product of $M_0^{\pm 1} ,\dots ,M_n^{\pm 1}$, 
  $g^{-1}$ maps $W^{\pm}$ as follows:
  \begin{align*}
    g^{-1}: \left\{
      \begin{array}{ll}
        W^{\pm} \to W^{\pm} & (\textrm{the number of $M_1^{\pm 1}$ and $M_2^{\pm 1}$ in $g$ is even}) \\
        W^{\pm} \to W^{\mp} & (\textrm{the number of $M_1^{\pm 1}$ and $M_2^{\pm 1}$ in $g$ is odd}) .
      \end{array}
    \right.
  \end{align*}
  By $M_0 W^{\pm} =W^{\pm}$, we thus obtain 
  \begin{align*}
    gM_0 g^{-1}: \left\{
      \begin{array}{ll}
        W^{\pm} \xrightarrow{g^{-1}} W^{\pm} \xrightarrow{M_0} W^{\pm} \xrightarrow{g} W^{\pm} 
        & (\textrm{the number of $M_1^{\pm 1}$ and $M_2^{\pm 1}$ in $g$ is even}) \\
        W^{\pm} \xrightarrow{g^{-1}} W^{\mp} \xrightarrow{M_0} W^{\mp} \xrightarrow{g} W^{\pm}
        & (\textrm{the number of $M_1^{\pm 1}$ and $M_2^{\pm 1}$ in $g$ is odd}). 
      \end{array}
    \right.
  \end{align*}
  Therefore, $W^{\pm}$ are non-trivial $\Ref$-subspaces, and 
  the proof is completed. 
\end{proof}

\begin{proof}[Proof of Lemma \ref{lem-red-2}]  
  For simplicity, we assume $\gamma_1 =-1$. 
  For each $(i_2 ,\dots ,i_n)$, we put 
  \begin{align*}
    h_{0,i_2,i_3,\dots ,i_n} &= \ee_{0,i_2,i_3,\dots ,i_n} 
    =\ee_0 \otimes \ee_{i_2} \otimes \ee_{i_3} \otimes \cdots \otimes \ee_{i_n} , \\
    h_{1,i_2,i_3,\dots ,i_n} &= 2\ee_{1,i_2,i_3,\dots ,i_n}- \ee_{0,i_2,i_3,\dots ,i_n} 
    =(2\ee_1 -\ee_0) \otimes \ee_{i_2} \otimes \ee_{i_3} \otimes \cdots \otimes \ee_{i_n} , \\
    f_{12;i_2 ,\dots ,i_n}^{\pm} &=h_{0,i_2,i_3,\dots ,i_n} \pm h_{1,i_2,i_3,\dots ,i_n}  .
  \end{align*}
  We consider proper subspaces 
  \begin{align*}
    W^{\pm} = \bigoplus_{(i_2 ,\dots ,i_n)} \CC f_{12;i_2 ,\dots ,i_n}^{\pm}
  \end{align*}
  of $\CC^{2^n}$ whose dimensions are $2^{n-1}$. 
  We show that these are non-trivial $\Ref$-subspaces. 
  Note that by the definition, we have 
  $$
  \ee_{1,\dots ,1}
  =\frac{1}{2}(h_{0,1,1,\dots ,1}+h_{1,1,1,\dots ,1})
  =\frac{1}{2}f_{12;1,\dots ,1}^{+} \in W^{+} .
  $$
  
  By arguments similar to those in Proof of Lemma \ref{lem-red-1}, we obtain 
  \begin{align*}
    M_1 W^{\pm} =W^{\mp},\quad 
    M_i W^{\pm} =W^{\pm} \ (i=2,\dots ,n). 
  \end{align*}

  We put 
  \begin{align*}
    \lambda_{i_2 ,\dots ,i_n} =(-1)^{n+i_2+\dots +i_n}
    \frac{(\alpha \beta +(-1)^{i_2+\dots +i_n} 
      \prod_{k=2}^n \gamma_k^{i_k})\prod_{k=2}^n \gamma_k^{1-i_k}}{\alpha \beta} .
  \end{align*}
  Since $\alpha +\beta=0$ by the assumption of the lemma, 
  $\lambda_{0 ,\dots ,0}$ is written as 
  \begin{align*}
    \lambda_{0 ,\dots ,0} =(-1)^{n}
    \frac{(\alpha \beta +1)\prod_{k=2}^n \gamma_k}{\alpha \beta}
    =(-1)^n \frac{(\alpha -1)(\beta -1)\prod_{k=2}^n \gamma_k}{\alpha \beta} .
  \end{align*}
  By Proposition \ref{prop-rep-matrix} and $\gamma_1 =-1$, we have  
  \begin{align*}
    M_0 \ee_{0,i_2 ,\dots ,i_n} 
    =\ee_{0,i_2 ,\dots ,i_n}+\lambda_{i_2 ,\dots ,i_n}\ee_{1,\dots ,1}  ,\quad
    M_0 \ee_{1,i_2 ,\dots ,i_n} 
    =\ee_{1,i_2 ,\dots ,i_n}+\lambda_{i_2 ,\dots ,i_n}\ee_{1,\dots ,1} , 
  \end{align*}
  and 
  \begin{align*}
    M_0 \cdot h_{0,i_2,\dots ,i_n} 
    =h_{0,i_2 ,\dots ,i_n}+\lambda_{i_2 ,\dots ,i_n}\ee_{1,\dots ,1}  , \quad 
    M_0 \cdot h_{1,i_2,\dots ,i_n} 
    =h_{1,i_2,\dots ,i_n}+\lambda_{i_2 ,\dots ,i_n}\ee_{1,\dots ,1} .
  \end{align*}
  These imply 
  \begin{align*}
    &M_0 \cdot f_{12;i_2,\dots ,i_n}^{-}=f_{12;i_2,\dots ,i_n}^{-}\in W^{-} , \\ 
    &M_0 \cdot f_{12;i_2,\dots ,i_n}^{+}
    =f_{12;i_2,\dots ,i_n}^{+}+2\lambda_{i_2 ,\dots ,i_n}\ee_{1,\dots ,1} \in W^{+} , 
  \end{align*}
  and hence we obtain $M_0 W^{\pm} =W^{\pm}$. 

  An argument similar to that in Proof of Lemma \ref{lem-red-1} shows that 
  $(gM_0 g^{-1})W^{\pm} =W^{\pm}$, 
  for each generator $gM_0 g^{-1}$ ($g\in \Mon$) of $\Ref$. 
  Therefore, $W^{\pm}$ are non-trivial $\Ref$-subspaces. 
  We conclude that the action of $\Ref$ is reducible. 
\end{proof}
%%%%%%%%%%%%%%%%%%%%%%%%%%%%%%%%%%%%%%%%%%%%%%%%%%%%%%%%%%%%%%%%%%%%%%%%%%%%%%%%%
%%% Theorem %%%%%%%%%%%%%%%%%%%%%%%%%%%%%%%%%%%%%%%%%%%%%%%%%%%%%%%%%%%%%%%%%%%
\begin{tm} \label{main3}
Assume that $\Ref$ acts on $\CC^{2^n}$ irreducibly. If the action of $\Ref^0$ is 
reducible, then $\Ref$ is finite {\rm (}and hence $\Mon$ is finite by 
{\rm Proposition \ref{finite-prop}} {\rm )}.
\end{tm}
To prove this, we use the following simple fact.
%%%%%%%%%%%%%%%%%%%%%%%%%%%%%%%%%%%%%%%%%%%%%%%%%%%%%%%%%%%%%%%%%%%%%%%%%%%%%%%%%%
%%% Lemma %%%%%%%%%%%%%%%%%%%%%%%%%%%%%%%%%%%%%%%%%%%%%%%%%%%%%%%%%%%%%%%%%%%
\begin{lem} \label{lemma for main3}
Let $G \subset \GL_n(\CC)$ be a subgroup acting on a 1-dimensional subspace 
$W \subset \CC^n$. Assume that a matrix $g \in \GL_n(\CC)$ normalizes $G$ and $gW \ne W$.  
If one of the followings holds, 
then $G$ acts on $W \oplus gW$ as scalar multiplications.
\\
{\rm (1)} \ $g$ is diagonalizable, and has two eigenvalues $\alpha_1$ and $\alpha_2$ 
such that $\alpha_1 \ne \pm \alpha_2$,
\\
{\rm (2)} \ $g$ is unipotent and $(g - E_n)^2 = 0$. 
\end{lem}
%%% proof %%%%%%%%%%%%%%%%%%%%%%%%%%%%%%%%%%%%%%%%%%%%%%%%%%%%%%%%%%%%%%%%%%%%%%%%%
\begin{proof} 
Since $G$ is normalized by $g$, we see that $G$ acts on $g^k W \ (k=1,2,\dots)$.   
In the cases of (1), we can write
\[
 W = \CC w, \quad w = w_1 + w_2,  
\]  
where $w_i \ne 0$ is an eigenvector of $g$ corresponding to $\alpha_i$. 
Then we have
\[
 g^2W = \CC \cdot (\alpha_1^2 w_1 + \alpha_2^2 w_2) \subset 
\CC w_1 \oplus \CC w_2 = W \oplus gW,  
\]
and $g^2W \ne W$ since $\alpha_1^2 \ne \alpha_2^2$. Therefore a 2-dimensional space 
$W \oplus gW$ contains three different $G$-invarinat subspaces 
$W, \ gW$ and $g^2W$, and this implies that $G$ acts on $W \oplus gW$ as constants. 
In the case of (2), we have
\[
 W = \CC w, \qquad gw = w + w', \quad gw' = w',
\] 
and 
\[
 g^2W = \CC \cdot (w + 2w') \subset 
\CC w \oplus \CC w' = W \oplus gW.
\]
By the same reason, $G$ acts on $W \oplus gW$ as constants.
\end{proof}
%%% proof %%%%%%%%%%%%%%%%%%%%%%%%%%%%%%%%%%%%%%%%%%%%%%%%%%%%%%%%%%%%%%%%%%%%%%%%
\begin{proof}[Proof of Theorem \ref{main3}]
Let us assume that the action of $\Ref^0$ is not irreducible, and 
let $W_0$ be a non-trivial irreducible $\Ref^0$-subspace. Then there exists a reflection 
$r_0 = gM_0g^{-1}$ such that $r_0 W_0 \ne W_0$. Replacing $W_0$ by $g^{-1}W_0$, 
we may assume that $r_0 = M_0$. By irreducibility, we have 
$W_0 \cap M_0W_0 = \mathbf{0}$ and $W_0$ does not contain any eigenvector of $M_0$. 
We see that $\dim W_0 = 1$ and $W_0 = \CC w_0$ with $w_0 \notin \ker N_0$ 
by Proposition \ref{prop-rep-matrix}. Note that $M_0^d \in \Ref^0$ for some $d$ since 
$\Ref/\Ref^0$ is finite, and hence we have $M_0^d W_0 = W_0$. Namely $w_0$ is an 
eigenvector of $M_0^d$, but not of $M_0$. If $\delta_0 = 1$, both of $M_0$ and $M_0^d$ have 
the unique eigenspace $\ker N_0$. Therefore we have $\delta_0 \ne 1$ and 
$\CC^{2^n} = \CC \ee_{1,\dots,1} \oplus \ker N_0$.   
Now we may assume that 
\[
 w_0 = \ee_{1,\dots,1} + \nu_0 \qquad 
(0 \ne \nu_0 \in \ker N_0), 
\] 
and $M_0^d W_0 = W_0$ implies $M_0^d = \mathrm{Id}$. Therefore the special eigenvalue 
$\delta_0$ is a $d$-th root of unity, and hence $\Ref^0 \subset \SL_{2^n}(\CC)$. 
Moreover, we may assume that $\gamma_2, \dots, \gamma_n \ne -1$ by Theorem \ref{main2}. 
\\ \indent
We show that $\Ref^0$ acts on $\CC \ee_{1,\dots,1}$, dividing into three cases.
 %%%%%%%%%%%%%%%%%%%%%%%%%%%%%%%%%%%%%%%%%%%%%%%%%%%%%%%%%%%%%%%%%%%
\\ \indent
(Case 1) \ Assume that $\delta_0 \ne -1$. Applying Lemma \ref{lemma for main3} 
for $G = \Ref^0$, $W = W_0$ and $g = M_0$, we see that $\Ref^0$ acts on
\[
 W_0 \oplus M_0 W_0 = \left< \ee_{1,\dots,1},\ \nu_0 \right>_{\CC} 
\]
as scalar multiplications. Therefore $\Ref^0$ acts on $\CC \ee_{1,\dots,1}$. 
\\ \indent %%%%%%%%%%%%%%%%%%%%%%%%%%%%%%%%%%%%%%%%%%%%%%%%%%%%%%%%%%%%%%%%%%%
(Case 2) \ Next we assume that $\delta_0 = -1$, that is,  
$M_0w_0 = -\ee_{1,\dots,1} + \nu_0$. We have the following two possibilities:
\\ 
(2-i) There is a $\Ref^0$-subspace $W_0' \not\subset \ker N_0$ different from 
$W_0$ and $M_0W_0$. 
\\ 
(2-ii) Any irreducible $\Ref^0$-subspace different from 
$W_0$ and $M_0W_0$ is contained in $\ker N_0$.
\\ \indent %%%%%%%%%%%%%%%%%%%%%%%%%%%%%%%%%%%%%%%%%%%%%%%%%%%%%%%%%%%%%%%%%%%
In the first case, $W_0'$ is generated by $\ee_{1,\dots,1} + \nu_0'$ with 
$\nu_0' \in \ker N_0$. If $\nu_0' = 0$, then $\Ref^0$ acts on $W_0' = \CC \ee_{1,\dots,1}$. 
If $\nu_0' \ne 0$, then the dimension of
\[
 W_0 + M_0W_0 + W_0' + M_0W_0' = \left< \ee_{1,\dots,1},\ \nu_0, \ \nu_0' \right>_{\CC} 
\]
is at most three, and we see that $\Ref^0$ acts on $W_0' = \CC \ee_{1,\dots,1}$ 
by the same argument with Lemma \ref{lemma for main3}. 
\\ \indent %%%%%%%%%%%%%%%%%%%%%%%%%%%%%%%%%%%%%%%%%%%%%%%%%%%%%%%%%%%%%%%%%%%%%%
Finally we consider the case (2-ii). Since $\CC^{2^n} = \CC \ee_{1,\dots,1} \oplus \ker N_0$,
 we can write
\begin{align*}
 M_n \nu_0 = c_n \ee_{1,\dots,1} + \nu_n, \qquad \ee_{1,\dots,1,0} = c_n' \ee_{1,\dots,1} + \nu_n'
\end{align*}
with $c_n, c_n' \in \CC$ and $\nu_n,\ \nu_n' \in \ker N_0$. Then we have
\begin{align*}
M_n (\pm \ee_{1,\dots,1} + \nu_0) 
&= \pm \frac{1}{\gamma_n} \ee_1 \otimes \cdots \otimes \ee_1 \otimes 
(- \ee_0 + \ee_1) + M_n \nu_0 \\
&= \pm \frac{1}{\gamma_n} (\ee_{1,\dots,1} - \ee_{1,\dots,1,0}) + c_n \ee_{1,\dots,1} + \nu_n\\
&= \pm \frac{1 \pm c_n \gamma_n}{\gamma_n} \ee_{1,\dots,1}
 \mp \frac{1}{\gamma_n} (c_n' \ee_{1,\dots,1} + \nu_n') + \nu_n \\
&= \pm \frac{1 \pm c_n \gamma_n - c_n'}{\gamma_n} \ee_{1,\dots,1}
 \mp \frac{1}{\gamma_n} \nu_n' + \nu_n.
\end{align*}
Let us assume that both of $M_nW_0 = W_0$ and $M_n(M_0W_0) = M_0W_0$ are hold. By the 
above calculation, we have
\begin{align*}
  \begin{cases}M_nW_0 = W_0 \\ M_n(M_0W_0) = M_0W_0 \end{cases}
&\Leftrightarrow 
 \begin{cases}
 M_n(\ee_{1,\dots,1} + \nu_0) = \text{const.} \times (\ee_{1,\dots,1} + \nu_0) \\ 
  M_n(-\ee_{1,\dots,1} + \nu_0) = \text{const.} \times (-\ee_{1,\dots,1} + \nu_0)
 \end{cases}
\\
&\Leftrightarrow 
 \begin{cases}
 \frac{1 + c_n \gamma_n - c_n'}{\gamma_n} \nu_0 = -\frac{1}{\gamma_n} \nu_n' + \nu_n\\
 \frac{-1 + c_n \gamma_n + c_n'}{\gamma_n} \nu_0 = \frac{1}{\gamma_n} \nu_n' + \nu_n
 \end{cases}
\Rightarrow \quad c_n \nu_0 = \nu_n.
\end{align*}
However, if $c_n \nu_0 = \nu_n$, we have
\[
 M_n \nu_0 = c_n \ee_{1,\dots,1} + \nu_n = c_n (\ee_{1,\dots,1} + \nu_0) = c_nw_0 \in W_0, 
\]
and hence $\nu_0 \in M_n^{-1}W_0 = W_0$. This contradicts $W_0 = \CC w_0$. 
Therefore, at least one of $M_nW_0 \ne W_0$ and $M_n(M_0W_0) \ne M_0W_0$ must be hold. 
Replacing $W_0$ by $M_0W_0$ if necessary, we assume 
that $M_n W_0 \ne W_0$. Applying Lemma \ref{lemma for main3} for $g = M_n$ 
(note that $M_n = E_2 \otimes \cdots \otimes E_2 \otimes G_n$ satisfies conditions 
for $g$ in Lemma \ref{lemma for main3} by the assumption $\gamma_n \ne -1$), 
we see that $\Ref^0$ acts on $W_0 \oplus M_nW_0$ as constants. 
Therefore, there are infinitely many $\Ref^0$-subspaces $W \subset W_0 \oplus M_nW_0$. 
Let $W$ be an irreducible $\Ref^0$-subspace different from $2^{n+1}$ subspaces
\[
M^IW_0, \quad M^I(M_0W_0), \qquad I \in \{0,1\}^n.
\]
Then $(M^I)^{-1} W$ is not equal to either of $W_0$ and $M_0W_0$, and hence 
$M_0$-invariant by the assumption. Therefore we have
\[
 R^I W = M^I M_0 (M^I)^{-1} W = M^I (M^I)^{-1} W = W, \qquad I \in \{0,1\}^n.
\]
By Lemma \ref{lem-NI}, we have $W = \CC \ff_J$ for some $J \in \{0,1\}^n$ and 
$\Ref^0$ acts on $(M^J)^{-1} W = \CC \ee_{1,\dots,1}$. 
%%%%%%%%%%%%%%%%%%%%%%%%%%%%%%%%%%%%%%%%%%%%%%%%%%%%%%%%%%%%%%%%%%%%%%%%%%%%%%
\\ \indent
From the above, $\CC \ee_{1,\dots,1}$ is an irreducible $\Ref^0$-subspace in any case. 
Now we see that each $\CC \ff_I \ (I \in \{0,1\}^n)$ is an irreducible $\Ref^0$-subspace 
since $M^I$ normalizes $\Ref^0$. However, if $\gamma_1 = -1$, then we have  
\begin{align*}
M_0 \ff_{1,0,\dots,0} &= M_0 (M_1 \ee_{1,\dots,1}) \\
&= M_0 (\ee_{0,1,\dots,1} - \ee_{1,\dots,1}) \\
&= \ee_{0,1,\dots,1} + (-1) \frac{(\alpha \beta +(-1)^{n-1} \prod_{k=2}^n \gamma_k)
          \gamma_1}{\alpha \beta}\ee_{1,\dots,1} - \delta_0 \ee_{1,\dots,1} \\
&= \ee_{0,1,\dots,1} - (\gamma_1 + 2 \delta_0)\ee_{1,\dots,1}  \\
&= \ff_{1,0,\dots,0} + \ee_{1,\dots,1} - (\gamma_1 + 2 \delta_0)\ee_{1,\dots,1} \\
&= \ff_{1,0,\dots,0} + (1 - \gamma_1 - 2 \delta_0)\ee_{1,\dots,1} 
= \ff_{1,0,\dots,0} + 2(1- \delta_0) \ff_{0,\dots,0}.
\end{align*}
This contradicts $\delta_0 \ne 1$, and therefore we have $\gamma_1 \ne -1$. 
By Lemma \ref{lemma for main3}, we see that $\Ref^0$ acts on 
$\CC \ee_{1,\dots,1} \oplus \CC M_i \ee_{1,\dots,1}$ as constants for $i=1,\dots,n$. 
Applying Lemma \ref{lemma for main3} again, for $W = \CC M_i \ee_{1,\dots,1}$ and $g = M_j$, 
we see that $\Ref^0$ acts on $\CC M_i \ee_{1,\dots,1} \oplus \CC M_jM_i \ee_{1,\dots,1}$ 
as constants for $1 \leq i < j \leq n$. Repeating this process, we can conclude that 
$\Ref^0$ acts on $\oplus_{I \in \{0,1\}^n} \CC \ff_I = \CC^{2^n}$, and $\Ref^0$ 
consists of scalar matrices. 
Since $\Ref^0 \subset \SL_{2^n}(\CC)$, we see that $\Ref^0$ is finite and so is $\Ref$.
\end{proof}
%%%%%%%%%%%%%%%%%%%%%%%%%%%%%%%%%%%%%%%%%%%%%%%%%%%%%%%%%%%%%%%%%%%%%%%%%%%%%%%%%%
%%% Corollary %%%%%%%%%%%%%%%%%%%%%%%%%%%%%%%%%%%%%%%%%%%%%%%%%%%%%%%%%%%%%%%%%%%%
\begin{cor} \label{main3-cor}
Assume that $\Ref$ acts on $\CC^{2^n}$ irreducibly {\rm(} that is, the 
parameters satisfy the conditions {\rm (irr$- \alpha \beta \gamma$)} and ``at most one of 
$\gamma_1, \dots, \gamma_n, \alpha \beta^{-1}$ is $-1$'' {\rm)}. 
If $\Mon$ is infinite {\rm(}for example, if $\delta_0 = 1${\rm)}, then 
$\Mon^0$ is irreducible. 
\end{cor}
%%%%%%%%%%%%%%%%%%%%%%%%%%%%%%%%%%%%%%%%%%%%%%%%%%%%%%%%%%%%%%%%%%%%%%%%%%%%%%%%% 
%%%%%%%%%%%%%%%%%%%%%%%%%%%%%%%%%%%%%%%%%%%%%%%%%%%%%%%%%%%%%%%%%%%%%%%%%%%%%%%%%%
\section{Double coverings arising from integral representations of $F_C$}
%% periods %%%%%%%%%%%%%%%%%%%%%%%%%%%%%%%%%%%%%%%%%%%%%%%%%%%%%%%%%%%%%%%%%%%%%%%
\subsection{Double coverings}
For $a = b = 1/2$ and $c_1 = \cdots = c_n = 1$, 
Lauricella's function $F_C(a,b,c;x)$ is a period of an algebraic variety
\begin{align*}
 V(x) \ : \ s^2 = t_1 \cdots t_n \big(1 - \sum_{i=1}^n t_i \big)
\big( t_1 \cdots t_n - \sum_{i=1}^n x_i \frac{t_1 \cdots t_n}{t_i} \big)
\end{align*}
with respect to a rational $n$-form $\displaystyle \omega = 
\frac{dt_1 \wedge \cdots \wedge dt_n}{s}$. The variety $V(x)$ is 
a double covering of $\PP^n$ branched along $n+2$ hyperplanes and a hypersurface 
of degree $n$. Similarly, Euler-type integrals of $F_C$ are regarded as periods of 
algebraic varieties that are cyclic branched coverings of projective spaces 
if all parameters are rational numbers.
\\ \indent
Note that the monodromy group is infinite and irreducible for the parameters 
$a = b = 1/2$ and $c_1 = \cdots = c_n = 1$. 
By Theorem \ref{main}, \ref{main2} and \ref{main3}, 
we see that the Zariski closure of $\Mon$ is $\mathbf{Sp}_{2^n}(\CC)$ if $n$ is odd, and 
$\mathbf{O}_{2^n}(\CC)$ if $n$ is even. Moreover, monodromy groups are defined over 
rational numbers. In the following, we study varieties $V(x)$ for $n=2$ and $3$. 
%%%%%%%%%%%%%%%%%%%%%%%%%%%%%%%%%%%%%%%%%%%%%%%%%%%%%%%%%%%%%%%%%%%%%%%%%%%%%%%%%%
\subsection{K3 surfaces}
It is classically known (e.g. \cite{Ba}) that Appell's $F_4$ satisfy the following 
formula 
\begin{align*}
 F_4(a,c+c'-a-1,c,c'; x(1-y), y(1-x)) \\
= {}_2F_1(a,c+c'-a-1,c;x) {}_2F_1(a,c+c'-a-1,c';y),
\end{align*}
and we see that $F_4(1/2,1/2,1,1; x(1-y), y(1-x))$ is a product of 
elliptic integrals. However it seems that a geometric proof of the formula 
is not known. In any way, we can show the following. 
%% Prop. %%%%%%%%%%%%%%%%%%%%%%%%%%%%%%%%%%%
\begin{prop}
For a general parameters $(x_1,x_2)$, the minimal smooth model of 
\[
 V(x_1,x_2) \ : \ s^2 = t_1 t_2 (1 - t_1 - t_2)(t_1 t_2 - x_1 t_2 - x_2 t_1)
\]
is a product Kummer surface with transcendental lattice 
$\mathrm{U}(2) \oplus \mathrm{U}(2)$ where 
$\mathrm{U}(2) = \begin{pmatrix} 0 & 2 \\ 2 & 0 \end{pmatrix}$.
\end{prop}
\noindent
%% Proof %%%%%%%%%%%%%%%%%%%%%%%%%%%%%%%%%%%
\begin{proof}
Let $V(x_1,x_2)$ be a double covering of $\PP^2$ branched along 
four lines and a conic:
\begin{align*}
  L_i = \{ T_i = 0 \} \quad (i=0,1,2), \qquad L_3 = \{ T_0 - T_1 - T_2 = 0 \},
\qquad Q = \{T_1T_2 - x_1 T_0T_2 - x_2 T_0T_1 = 0\}.
\end{align*} 
For a general parameters $(x_1,x_2)$, it has $A_1$-singularities at points over 
\[
 P_{i3} = L_i \cap L_3 \quad (i=0,1,2), \qquad \{P', \ P'' \} = Q \cap L_3
\]
and $D_4$-singularities at points over
\[
 P_{01} = L_0 \cap L_1 \cap Q = [0:0:1], \quad 
P_{02} = L_0 \cap L_2 \cap Q = [0:1:0], \quad 
P_{12} = L_1 \cap L_2 \cap Q= [1:0:0].
\]
Hence it is a double sextic with rational singularities, and the minimal 
resolution $S = S(x_1, x_2)$ is a K3 surface. 
Let us consider a pencil of lines passing through $P'$. 
For such a line $\ell$, let $P_{Q}(\ell)$ be another intersection point with $Q$.
The (strict) pull back $\pi^{-1} \ell$ by the projection 
$\pi : S \rightarrow \PP^2$ is a double covering of $\ell$ branched over 
four points $\ell \cap L_i \ (i=0,1,2)$ and $P_{Q}(\ell)$. 
Therefore they form an elliptic fibration with $2$-torsion sections 
$\pi^{-1}L_i \ (i=0,1,2)$  and $\pi^{-1}Q$. 
%%%%% elliptic pencil %%%%%%%%%%%%%%%%%%%%%%%%%%%%%%%%%%%%%%%%%%%%%%%%%%%%%%%%%%%%% 
\begin{figure}[htbp] \begin{center}
\begin{overpic}[width=8cm]{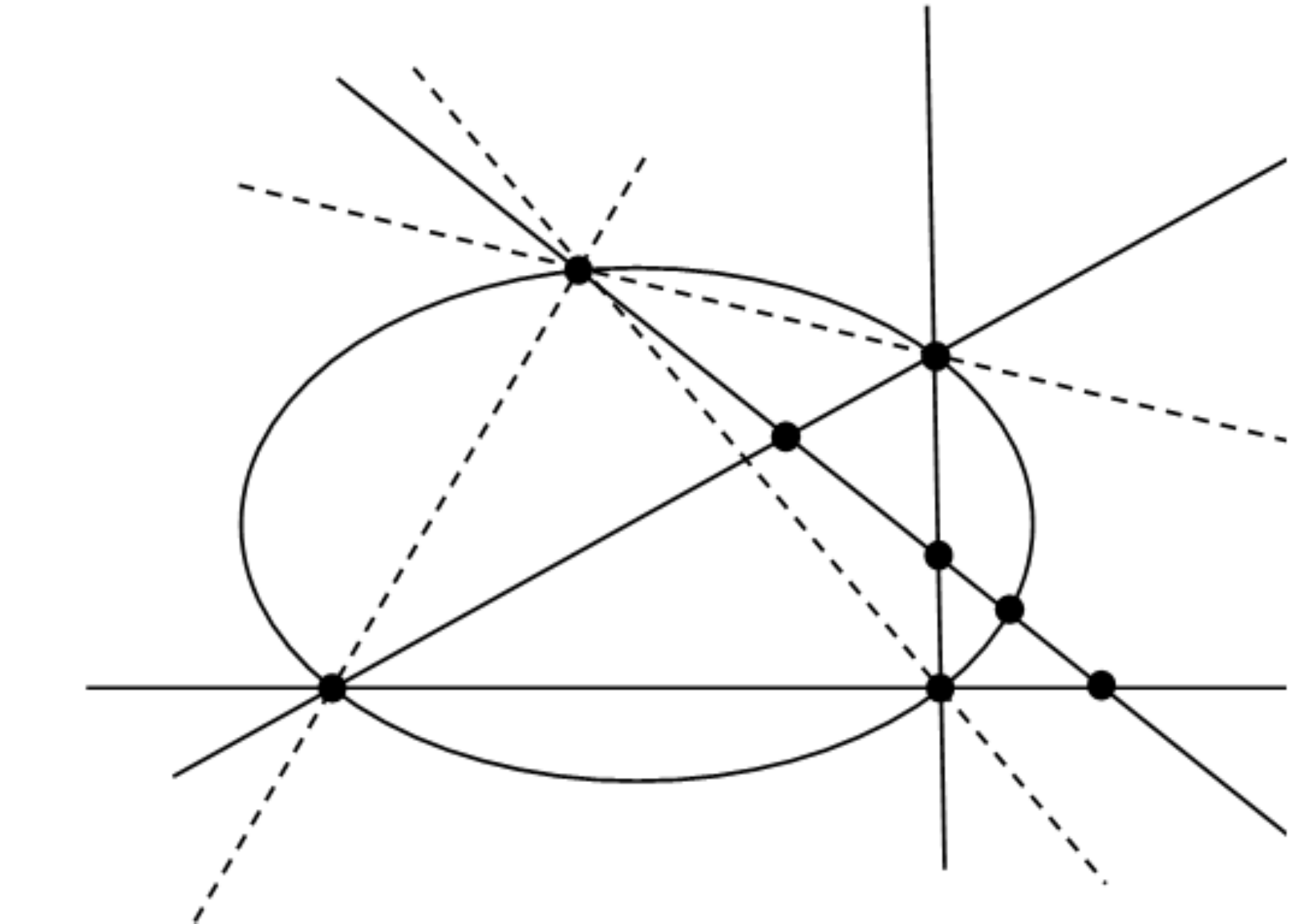}
 \put(0,17){$L_0$} \put(70,0){$L_1$} \put(100,60){$L_2$} \put(100,5){$L_3$}
 \put(18,0){$\ell_{02}$} \put(85,0){$\ell_{01}$} \put(100,35){$\ell_{12}$}
 \put(50,7){$Q$}
 \put(42,42){$P'$}
\end{overpic}
\end{center} \end{figure}
%%%%%%%%%%%%%%%%%%%%%%%%%%%%%%%%%%%%%%%%%%%%%%%%%%%%%%%%%%%%%%%%%%%%%%%%%%%%%%%%%%
\\
Let $\ell_{ij} \ (i,j = 0,1,2)$ be the line passing through $P_{ij}$ and $P'$. 
These three lines and $L_3$ gives four $I_0^*$-fibers, and we obtain disjoint sixteen 
smooth rational curves from their components. Hence $S$ is a Kummer surface. Let 
$\mathrm{NS}(S)$ be the N\'eron-Severi group of $S$, and 
$\rho(S) = \mathrm{rank} \ \mathrm{NS}(S)$ be the Picard number.
Since the family has 2-dimensional moduli, we have $\rho(S) \leq 18$ for a general 
member $S$. By the Shioda-Tate formula (\cite{SS}, Corollary 6.13), we see that 
$\rho(S) = 18$ and the Mordell-Weil rank is zero. Moreover, the Mordell-Weil group 
is precisely $(\ZZ/2\ZZ)^2$, since the specialization of torsion sections to a singular 
fiber is injective and we have only $I_0^*$-fibers. By the formula (22) in \cite{SS}, 
we have
\[
 \mathrm{disc} \ \mathrm{NS}(S) 
= -(\mathrm{disc} \ D_4)^4 / |(\ZZ/2\ZZ)^2|^2 = -4^4 / 4^2 = -16,
\]
where $D_4$ is the Dynkin lattice of type $D_4$. Therefore the discriminant of 
the transcendental lattice $T_S$ is $16$. On the other hand, $T_S$ must be of 
the form $\mathrm{U}(2) \oplus T'(2)$ where $T'$ is a even lattice 
(\cite{Mo}, Corollary 4.4). Hence we have $T' = \mathrm{U}$ and $S$ is a product 
Kummer surface. 
\end{proof}
%%%%%%%%%%%%%%%%%%%%%%%%%%%%%%%%%%%%%%%%%%%%%%%%%%%%%%%%%%%%%%%%%%%%%%%%%%%%%%%%%%
Changing our basis by the following matrix $P$, we have new intersection matrix 
$H' = {}^tPHP$ and monodromy representations $M_k' = P^{-1}M_kP \ (k=0,1,2)$:
\[
 P = \left(
\begin{array}{cccc}
 \frac{1}{2} & 0 & 0 & -\frac{1}{2} \\ 
 0 & 1 & 0 & 0 \\
 0 & 0 & 1 & 0 \\ 
 0 & 0 & 0 & 2 \\
\end{array}
\right), 
\qquad
H' = \left(
\begin{array}{cccc}
 0 & 0 & 0 & 1 \\
 0 & 0 & -1 & 0 \\
 0 & -1 & 0 & 0 \\
 1 & 0 & 0 & 0 \\
\end{array}
\right),
\]
\begin{align*}
 M_0' = \left(
\begin{array}{cccc}
 0 & 0 & 0 & -1 \\
 0 & 1 & 0 & 0 \\
 0 & 0 & 1 & 0 \\
 -1 & 0 & 0 & 0 \\
\end{array}
\right),
\quad
M_1' = \left(
\begin{array}{cccc}
 1 & -2 & 0 & 0 \\
 0 & 1 & 0 & 0 \\
 0 & 0 & 1 & -2 \\
 0 & 0 & 0 & 1 \\
\end{array}
\right),
\quad 
M_2' = \left(
\begin{array}{cccc}
 1 & 0 & -2 & 0 \\
 0 & 1 & 0 & -2 \\
 0 & 0 & 1 & 0 \\
 0 & 0 & 0 & 1 \\
\end{array}
\right).
\end{align*}
%%%%%%%%%%%%%%%%%%%%%%%%%%%%%%%%%%%%%%%%%%%%%%%%%%%%%%%%%%%%%%%%%%%%%%%%%%%%%%%%%% 
Now let us consider the Segre embedding
\begin{align*}
 \PP^1 \times \PP^1 \longrightarrow \PP^3, \qquad
[s_0 : s_1] \times [t_0 : t_1] \mapsto [s_0t_0 : s_0t_1 : s_1t_0 : s_1t_1]. 
\end{align*}
The image satisfies a quadratic relation
\[
 (s_0t_0, s_0t_1, s_1t_0, s_1t_1) H' {}^t(s_0t_0, s_0t_1, s_1t_0, s_1t_1) = 0
\]
and $M_k'$ acts on $\PP^1 \times \PP^1$ by
\begin{align*}
M_0' \cdot (s, t) = (-t^{-1}, -s^{-1}), \quad
M_1' \cdot (s, t) = (s, t + 2), \quad
M_2' \cdot (s, t) = (s + 2, t)
\end{align*}
where $s = s_1/s_0$ and $t = t_1/t_0$. Since the congruence subgroup 
$\Gamma(2) \subset \SL_2(\ZZ)$ is generated by 
$\begin{pmatrix} 1 & 2 \\ 0 & 1\end{pmatrix}$ and
$\begin{pmatrix} 1 & 0 \\ 2 & 1\end{pmatrix}$ projectively, we see that a subgroup 
of index 2, generated by
\[
M_1', \ M_2', \ M_0'M_1M_0', \ M_0'M_2M_0',
\]
is isomorphic to $\Gamma(2) \times \Gamma(2)$ as projective transformations. 
%%%%%%%%%%%%%%%%%%%%%%%%%%%%%%%%%%%%%%%%%%%%%%%%%%%%%%%%%%%%%%%%%%%%%%%%%%%%%%%%
\begin{rem}
{\rm (1)} \ The projective monodromy of ${}_2F_1(1/2,1/2,1)$ is 
$\Gamma(2) / \{\pm1\}$. 
\\
{\rm (2)} \ The product $\SL_2(\CC) \times \SL_2(\CC)$ is a double cover of 
$\SO_4(\CC)$.  
\end{rem}
%%%%%%%%%%%%%%%%%%%%%%%%%%%%%%%%%%%%%%%%%%%%%%%%%%%%%%%%%%%%%%%%%%%%%%%%%%%%%%%%%%
\subsection{Calabi-Yau varieties}
%%%%%%%%%%%%%%%%%%%%%%%%%%%%%%%%%%%%%%%%%%%%%%%%%%%%%%%%%%%%%%%%%%%%%%%%%%%%%%%%%%
%% Prop. %%%%%%%%%%%%%%%%%%%%%%%%%%%%%%%%%%%%%%%%%%%%%%%%%%%%%%%%%%%%%%%%%%%%%%%%%
\begin{prop}
For a general parameter $x = (x_1,x_2,x_3)$, we have a resolution 
$\widetilde{V} = \widetilde{V(x)}$ of  
\[
 V(x) \ : \ s^2 = t_1 t_2 t_3 (1 - t_1 - t_2 - t_3)
(t_1 t_2 t_3 - x_1 t_2 t_3 - x_2 t_1 t_3 - x_3 t_1 t_2)
\]
which is a Calabi Yau 3-fold with Hodge numbers $h^{1,1}(\widetilde{V}) = 68, \ 
h^{2,1}(\widetilde{V}) = 4$ and 
the Euler characteristic $e(\widetilde{V}) = 128$. 
\end{prop}
%% Proof %%%%%%%%%%%%%%%%%%%%%%%%%%%%%%%%%%%%%%%%%%%%%%%%%%%%%%%%%%%%%%%%%%%%%
\begin{proof}
The variety $V = V(x)$ is a double covering of $\PP^3$ branched along the following 
five planes $H_i$ and a nodal cubic surface $S$: 
\begin{align*}
H_i : T_i = 0 \quad (i=0,1,2,3), \qquad H_4 : (T_0 - T_1 - T_2 - T_3) = 0, \\ 
S : T_1 T_2 T_3 - T_0(x_1 T_2 T_3 + x_2 T_1 T_3 + x_3 T_1 T_2) = 0. 
\end{align*}
The cubic surface $S$ is known as the Cayley cubic, and it has four nodes 
\[
P_0 = [1:0:0:0], \quad P_1 = [0:1:0:0], \quad P_2 = [0:0:1:0], \quad P_3 = [0:0:0:1].
\]
The branch divisor $B = H_0 + \cdots + H_4 + S$ has singularities as given 
in the table below. 
%% singular locus %%%%%%%%%%%%%%%%%%%%%%%%%%%%%%%%%%%%%%%%%%%%%%%%%%%%%%%%%%%
\[
\begin{array}{|c||c|} 
\hline 
\text{5-fold points} & 
P_i = H_j \cap H_k \cap H_l \cap S, \quad (\{i,j,k,l\}=\{0,1,2,3\})
\\ \hline
%%%%%%%%%
\text{4-fold points} &
H_i \cap H_j \cap H_4 \cap S, \quad (i,j = 0,1,2,3)
\\ \hline
%%%%%%%%%
\text{triple lines} &
L_{ij} = H_i \cap H_j \cap S \quad (i,j = 0,1,2,3)
\\ \hline
%%%%%%%%%%
\text{double curve} &
S \cap H_4 \ (\text{smooth cubic curve}), 
\quad L_{i4} = H_i \cap H_4 \quad (i = 0,1,2,3)
\\ \hline
\end{array}
\]
%%%%%%%%%%%%%%%%%%%%%%%%%%%%%%%%%%%%%%%%%%%%%%%%%%%%%%%%%%%%%%%%%%%%%%%%%%%%%%%%
We resolve them in three steps by admissible blow-ups in \cite{CS}, 
according to Cynk and Szemberg. 
(However $B$ is not an arrangement in the sense of \cite{CS}, and we can not apply 
formulas in \cite{CS} for the pair $(\PP^3, B)$ directly). 
\\
(step 1) \ Let $\sigma_1 : U_1 \rightarrow \PP^3$ be the blow-up 
at $5$-fold points $P_0, P_1, P_2$ and $P_3$, and let $E_i$ be the exceptional 
divisor corresponding to $P_i$. We denote the strict transform of a subvariety 
$D \subset \PP^3$ by $D^{(1)}$, and take  
$B_1 = \sum H_k^{(1)} + S^{(1)} + \sum E_k$. 
On $E_0 \cong \PP^2$, three lines 
$H_k^{(1)} \cap E_0 \ (k=1,2,3)$ form a triangle with a circumscribed 
conic $S^{(1)} \cap E_0$. The same is true for other $E_i$. 
%%%%% 5-fold point %%%%%%%%%%%%%%%%%%%%%%%%%%%%%%%%%%%%%%%%%%%%%%%%%%%%%%%%%%%%% 
\begin{figure}[htbp] \begin{center}
\begin{overpic}[width=12cm]{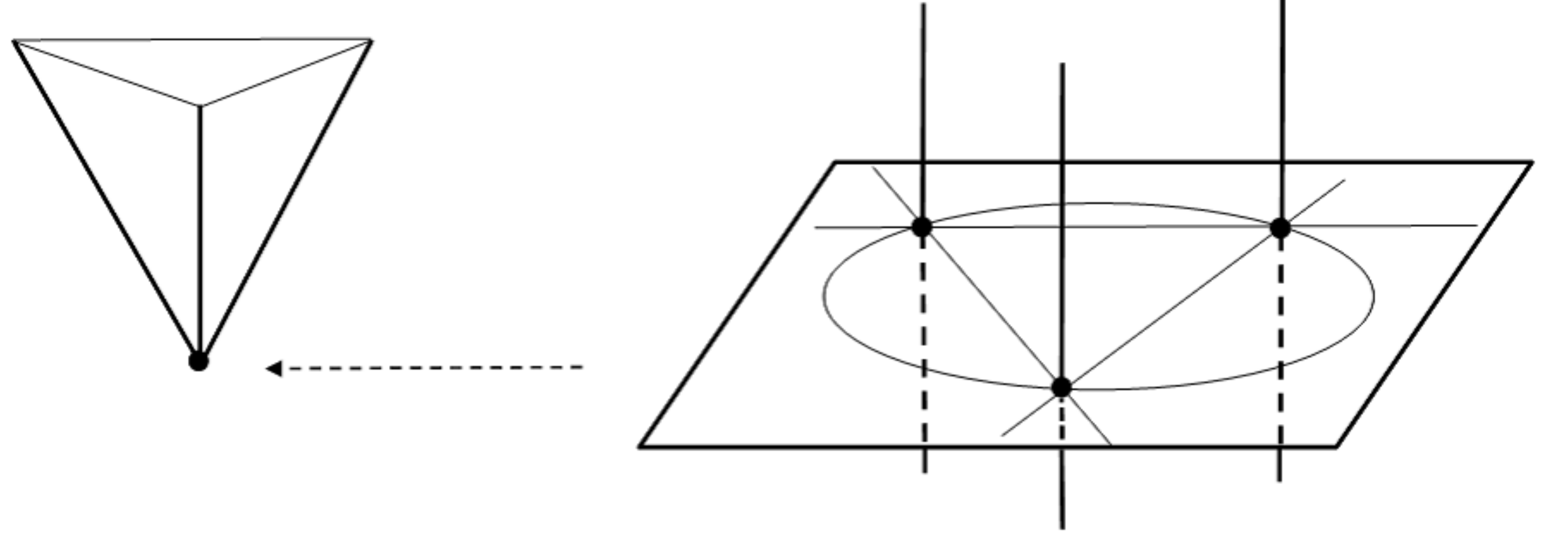}
 \put(11,7){$P_0$} \put(25,13){$\sigma_1$}  
 \put(92,12){$E_0 \cong \PP^2$} 
 \put(53,33){$L_{12}^{(1)}$} \put(62,28){$L_{23}^{(1)}$} \put(76,33){$L_{13}^{(1)}$}
\end{overpic}
\end{center} \end{figure}
%%%%%%%%%%%%%%%%%%%%%%%%%%%%%%%%%%%%%%%%%%%%%%%%%%%%%%%%%%%%%%%%%%%%%%%%%%%%%%%%%
\\
Now triple lines $L_{ij}^{(1)}$ are disjoint, and there are new twelve 
$4$-fold points as intersections of $L_{ij}^{(1)}$ and $E_k$. 
%%%%%%%%%%%%%%%%%%%%%%%%%%%%%%%%%%%%%%%%%%%%%%%%%%%%%%%%%%%%%%%%%%%%%%%%%%%%%%%
(step 2) \  Let $\sigma_2 : U_2 \rightarrow U_1$ be the blow-up along six triple 
lines $L_{ij}^{(1)}$. Let $E_{ij}$ be the exceptional divisor corresponding to 
$L_{ij}^{(1)}$,  that are $\PP^1$-bundles over $\PP^1$. 
%%%%% triple line %%%%%%%%%%%%%%%%%%%%%%%%%%%%%%%%%%%%%%%%%%%%%%%%%%%%%%%%%%%%% 
\begin{figure}[htbp] \begin{center}
\begin{overpic}[width=10cm]{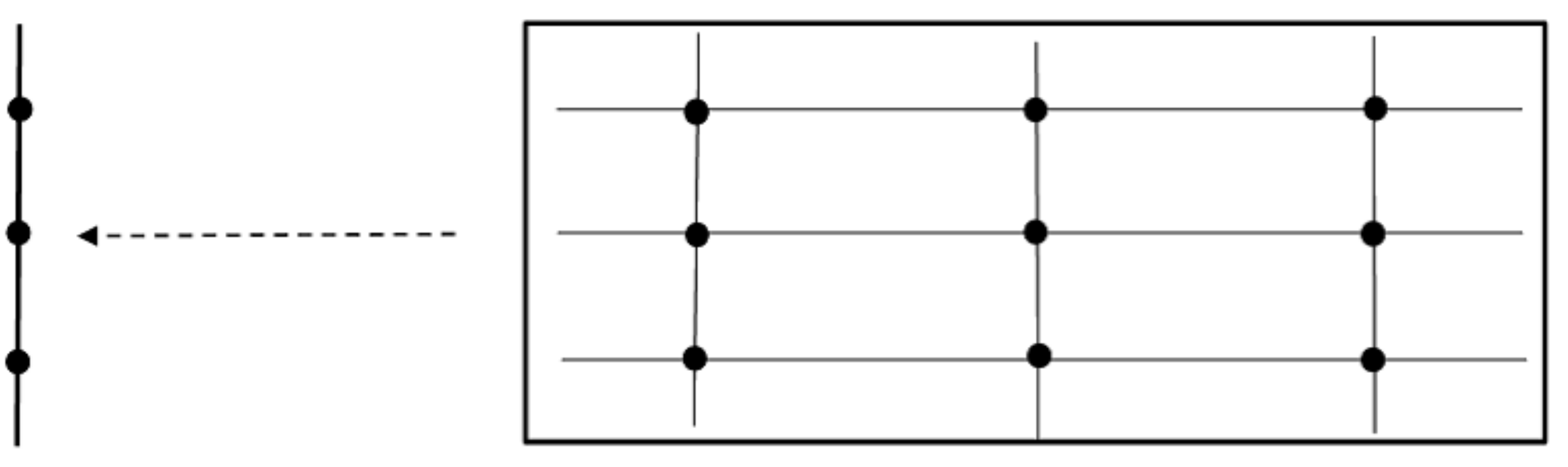}
\small
 \put(0,-2){$L_{12}^{(1)}$}   
 \put(-14,6){$E_0 \cap L_{12}^{(1)}$} 
 \put(-14,14){$E_3 \cap L_{12}^{(1)}$} 
 \put(-17,22){$H_4^{(1)} \cap L_{12}^{(1)}$}
 \put(17,17){$\sigma_2$}
 \put(55,30){$E_{12}$}  
 \put(37,-2){$H_1^{(2)} \cap E_{12}$} 
 \put(59,-2){$H_2^{(2)} \cap E_{12}$} \put(80,-2){$S^{(2)} \cap E_{12}$}
 \put(100,6){$E_0^{(2)} \cap E_{12}$} 
 \put(100,14){$E_3^{(2)} \cap E_{12}$} 
 \put(100,22){$H_4^{(2)} \cap E_{12}$}
\end{overpic}
\end{center} \end{figure}
\\
\normalsize
%%%%%%%%%%%%%%%%%%%%%%%%%%%%%%%%%%%%%%%%%%%%%%%%%%%%%%%%%%%%%%%%%%%%%%%%%%%%%%%%%
At this point, the branch divisor
\[
 B_2 = \sum H_k^{(2)} + S^{(2)} + \sum E_k^{(2)} + \sum E_{ij}
\]
is normal crossing, and $H_k^{(2)} \ (k=0,1,2,3)$ and $S^{(2)}$ are disjoint.
\\
%%%%%%%%%%%%%%%%%%%%%%%%%%%%%%%%%%%%%%%%%%%%%%%%%%%%%%%%%%%%%%%%%%%%%%%%%%%%%%%%%
(step 3) \  Let $\sigma_3 : U_3 \rightarrow U_2$ be a blow-up along double curves 
of $B_2$. This is not unique and depends on the order of blow-ups 
of double curves. However it does not affect on Euler characteristics of 
resultant varieties. We blow up in the following order:
\\
(i) blow up along $E_{ij} \cap H_i^{(2)} \ (i \ne 4)$ and $E_{ij} \cap S^{(2)}$,
\\
(ii) blow up along $E_{ij} \cap H_4^{(2)}$ and $E_{ij} \cap E_k^{(2)}$,
\\
(iii) blow up along $H_i^{(2)} \cap H_4^{(2)}$, \ $H_i^{(2)} \cap E_k^{(2)}$ \ 
$(i \ne 4, k)$, \ $S^{(2)} \cap H_4^{(2)}$ and  $S^{(2)} \cap E_k^{(2)}$,
\\
and put
\[
 B_3 = \sum H_k^{(3)} + S^{(3)} + \sum E_k^{(3)} + \sum E_{ij}^{(3)}.
\]
%% branch divisor %%%%%%%%%%%%%%%%%%%%%%%%%%%%%%%%%%%%%%%%%%%%%%%%%%%%%%%%%%%
The Euler characteristics of $U_i$ and components of $B_i$ are changed as 
in the following table.
\[
\begin{array}{|c||c|c|c|c|c|c|} 
\hline 
& U_i & H_i \ (i \ne 4) & H_4 & S & E_i & E_{ij}
\\ \hline
%%%%%%%%%
U_0 = \PP^3 & 4 & 3 & 3 & 5 & - & -
\\ \hline
%%%%%%%%%
U_1 & 12 & 6 & 3 & 9 & 3 & -
\\ \hline
%%%%%%%%%%
U_2 & 24 & 6 & 9 & 9 & 6 & 4
\\ \hline
%%%%%%%%%%
U_3 - \rm{(i)} & 60 & 6 & 27 & 9 & 15 & 4
\\ \hline
%%%%%%%%%%
U_3 - \rm{(ii)} & 96 & 6 & 27 & 9 & 15 & 4
\\ \hline
%%%%%%%%%%
U_3 - \rm{(iii)} & 136 & 6 & 27 & 9 & 15 & 4
\\ \hline
\end{array}
\]
Let $\pi : \widetilde{V} \rightarrow U_3$ be the double covering branched along 
$B_3$. Then $\widetilde{V}$ is a Calabi-Yau variety with the Euler characteristic 
$e(\widetilde{V}) = 128$, and hence 
$h^{1,1}(\widetilde{V}) - h^{1,2}(\widetilde{V}) = 64$. 
\\ \indent
%%%%%%%%%%%%%%%%%%%%%%%%%%%%%%%%%%%%%%%%%%%%%%%%%%%%%%%%%%%%%%%%%%%%%%%%%%%%%%%%%
Next we compute Hodge numbers. By Proposition 2.1 in \cite{CvS}, we have
 \[
 \mathrm{H}^1(\widetilde{V}, \Theta_{\widetilde{V}}) \cong 
\mathrm{H}^1(U_3, \Theta_{U_3}(\log B_3))
\oplus \mathrm{H}^1(U_3, \Theta_{U_3} \otimes \mathcal{L}^{-1})
\]
where $\Theta_X$ is the tangent bundle, $\Theta_X(\log D)$ is the sheaf of 
logarithmic vector field (\cite{CvS}, \cite{EV}) and 
$\mathcal{L}^{\otimes 2} \cong \mathcal{O}_{U_3}(B_3)$. Moreover 
$\mathrm{H}^1(U_3, \Theta_{U_3}(\log B_3))$ is isomorphic to the space of equisingular 
deformations of $B$ in $\PP^3$, and $h^1(\Theta_{U_3} \otimes \mathcal{L}^{-1})$ 
is the sum of genera of all blown-up curves (see \cite{CvS}).  Since 
blown-up curves are rational except an elliptic curve $S \cap H_4$, we have
$h^1(\Theta_{U_3} \otimes \mathcal{L}^{-1}) = 1$. Let us show 
$h^1(\Theta_{U_3}(\log B_3))=3$. (Then we have $h^1(\Theta_{\widetilde{V}}) = 4$, and
we can conclude that $h^{1,2}(\widetilde{V}) = 4$ by the Serre duality since 
$K_{\widetilde{V}} \cong \mathcal{O}_{\widetilde{V}}$.) To show this, let $W$ be an octic 
surface which has similar singularities with $V(x)$. By a projective transformation, 
we may assume that 5-folds points of $W$ are $P_0, \dots, P_3$. Consequently    
triple lines of $W$ must be  $L_{ij}$. Since $W$ has multiplicity $5$ at $P_i$, 
the polynomial $F(T_0,\dots,T_3)$ defining $W$ is a linear combination of
\[
 T_i^3 T_j^3 T_k^2, \quad T_i^3 T_j^3 T_k T_l, \quad T_i^3 T_j^2 T_k^2 T_l, \quad
T_0^2 T_1^2 T_2^2 T_3^2, \qquad \{i,j,k,l\} =\{0,1,2,3\}. 
\]
Moreover $F$ belongs to ideals $(T_i^3, T_i^2T_j, T_iT_j^2, T_j^3)$ since 
$F$ vanishes on $L_{ij}$ with third order. Therefore $F$ does not have terms  
$T_i^3 T_j^3 T_k^2, \ T_i^3 T_j^3 T_k T_l$, and we have $F = T_0 T_1 T_2 T_3 G$
where $G$ is a linear combination of
\[
 T_i^2 T_j T_k, \quad T_0 T_1 T_2 T_3, \qquad \{i,j,k,l\} =\{0,1,2,3\}.
\]
Then the singular locus of a quartic $G=0$ must contain an elliptic curve $C$ 
of degree 3 (a deformation of $S \cap H_4$). Note that $C$ is on a certain plane 
$H$. Since 4-fold points of $W$ are on $C$, double lines connecting 4-fold 
points (a deformation of $L_{i4}$) must be on $H$. This implies that $G = 0$ 
decomposes into $H$ and a cubic surface. We see that $W$ coming from our branch 
divisors $B \subset \PP^3$, and we have $h^1(\Theta_{U_3}(\log B_3))=3$.
\end{proof} 
%%%%%%%%%%%%%%%%%%%%%%%%%%%%%%%%%%%%%%%%%%%%%%%%%%%%%%%%%%%%%%%%%%%%%%%%%%%%%%%%%%
%%%%%%%%%%%%%%%%%%%%%%%%%%%%%%%%%%%%%%%%%%%%%%%%%%%%%%%%%%%%%%%%%%%%%%%%%%%%%%%%%% 
Changing our basis by the following matrix $P$ as in the case $n=2$, 
we have $H' = {}^tPHP$ and $M_k' = P^{-1}M_kP \ (k=0,1,2,3)$:
\begin{align*}
\small
P = \left(
\begin{array}{cccccccc}
 \frac{1}{2} & 0 & 0 & -\frac{1}{2} & 0 & -\frac{1}{2} & -\frac{1}{2} & 0 \\
 0 & -1 & 0 & 0 & 0 & 0 & 0 & 0 \\
 0 & 0 & -1 & 0 & 0 & 0 & 0 & 0 \\
 0 & 0 & 0 & 1 & 0 & 0 & 0 & 0 \\
 0 & 0 & 0 & 0 & 1 & 0 & 0 & 0 \\
 0 & 0 & 0 & 0 & 0 & 1 & 0 & 0 \\
 0 & 0 & 0 & 0 & 0 & 0 & 1 & 0 \\
 0 & 0 & 0 & 0 & 0 & 0 & 0 & 2 \\
\end{array}
\right),
\quad
H' = \left(
\begin{array}{cccccccc}
 0 & 0 & 0 & 0 & 0 & 0 & 0 & 1 \\
 0 & 0 & 0 & 0 & 0 & 0 & 1 & 0 \\
 0 & 0 & 0 & 0 & 0 & 1 & 0 & 0 \\
 0 & 0 & 0 & 0 & 1 & 0 & 0 & 0 \\
 0 & 0 & 0 & -1 & 0 & 0 & 0 & 0 \\
 0 & 0 & -1 & 0 & 0 & 0 & 0 & 0 \\
 0 & -1 & 0 & 0 & 0 & 0 & 0 & 0 \\
 -1 & 0 & 0 & 0 & 0 & 0 & 0 & 0 \\
\end{array}
\right),
\\
M_0' = \left(
\begin{array}{cccccccc}
 1 & 0 & 0 & 0 & 0 & 0 & 0 & 0 \\
 0 & 1 & 0 & 0 & 0 & 0 & 0 & 0 \\
 0 & 0 & 1 & 0 & 0 & 0 & 0 & 0 \\
 0 & 0 & 0 & 1 & 0 & 0 & 0 & 0 \\
 0 & 0 & 0 & 0 & 1 & 0 & 0 & 0 \\
 0 & 0 & 0 & 0 & 0 & 1 & 0 & 0 \\
 0 & 0 & 0 & 0 & 0 & 0 & 1 & 0 \\
 1 & 0 & 0 & 0 & 0 & 0 & 0 & 1 \\
\end{array}
\right),
\quad
M_1' = \left(
\begin{array}{cccccccc}
 1 & 2 & 0 & 0 & 0 & 0 & 0 & -2 \\
 0 & 1 & 0 & 0 & 0 & 0 & 0 & 0 \\
 0 & 0 & 1 & 1 & 0 & 0 & 0 & 0 \\
 0 & 0 & 0 & 1 & 0 & 0 & 0 & 0 \\
 0 & 0 & 0 & 0 & 1 & -1 & 0 & 0 \\
 0 & 0 & 0 & 0 & 0 & 1 & 0 & 0 \\
 0 & 0 & 0 & 0 & 0 & 0 & 1 & -2 \\
 0 & 0 & 0 & 0 & 0 & 0 & 0 & 1 \\
\end{array}
\right),
\\
M_2' = \left(
\begin{array}{cccccccc}
 1 & 0 & 2 & 0 & 0 & 0 & 0 & -2 \\
 0 & 1 & 0 & 1 & 0 & 0 & 0 & 0 \\
 0 & 0 & 1 & 0 & 0 & 0 & 0 & 0 \\
 0 & 0 & 0 & 1 & 0 & 0 & 0 & 0 \\
 0 & 0 & 0 & 0 & 1 & 0 & -1 & 0 \\
 0 & 0 & 0 & 0 & 0 & 1 & 0 & -2 \\
 0 & 0 & 0 & 0 & 0 & 0 & 1 & 0 \\
 0 & 0 & 0 & 0 & 0 & 0 & 0 & 1 \\
\end{array}
\right),
\quad
M_3' = \left(
\begin{array}{cccccccc}
 1 & 0 & 0 & 0 & -2 & 0 & 0 & -2 \\
 0 & 1 & 0 & 0 & 0 & 1 & 0 & 0 \\
 0 & 0 & 1 & 0 & 0 & 0 & 1 & 0 \\
 0 & 0 & 0 & 1 & 0 & 0 & 0 & -2 \\
 0 & 0 & 0 & 0 & 1 & 0 & 0 & 0 \\
 0 & 0 & 0 & 0 & 0 & 1 & 0 & 0 \\
 0 & 0 & 0 & 0 & 0 & 0 & 1 & 0 \\
 0 & 0 & 0 & 0 & 0 & 0 & 0 & 1 \\
\end{array}
\right).
\end{align*}
\normalsize
These give an integral representation of $\Mon$, but we do not know 
arithmeticy of $\Mon$, that is, the finiteness of the index $|\Sp(H',\ZZ) : \Mon|$. 
%%%%%%%%%%%%%%%%%%%%%%%%%%%%%%%%%%%%%%%%%%%%%%%%%%%%%%%%%%%%%%%%%%%%%%%%%%%%%%%%%% 

%%%%%%%%%%%%%%%%%%%%%%%%%%%%%%%%%%%%%%%%%%%%%%%%%%%%%%%%%%%%%%%%%%%%%%%%%%%%%%%% 
%\affiliationone{
   Yoshiaki Goto \\
   General Education \\
   Otaru University of Commerce \\
   Midori 3-5-21, Otaru, Hokkaido, 047-8501 \\
   Japan \\
   \texttt{goto@res.otaru-uc.ac.jp}
%}
\bigskip
\\ 
%\affiliationtwo{
   Kenji Koike \\
   Faculty of Education \\ 
   University of Yamanashi \\
   Takeda 4-4-37, Kofu, Yamanashi, 400-8510 \\
   Japan \\
   \texttt{kkoike@yamanashi.ac.jp}
%}
%
\end{document}